\documentclass[11pt]{amsart}

\usepackage{graphicx}
\usepackage{enumitem}
\usepackage{lipsum}
\usepackage{verbatim}
\usepackage{url}
\usepackage{amssymb}
\usepackage{amsthm}
\usepackage{amsmath}
\usepackage[all, cmtip]{xy}
\usepackage{tikz}
\usepackage{dsfont}
\usepackage{spectralsequences}
\usepackage{mathtools}
\usetikzlibrary{matrix}
\usepackage[center]{caption}
\usepackage{stmaryrd}
\usepackage{bbm}
\usepackage{changepage}
\graphicspath{{pics/}}

\calclayout

\newlength\mylen
\newlist{mycases}{enumerate}{1}
\setlist[mycases,1]{label=\textbf{Case~\arabic*.}, 
  labelwidth=\dimexpr-\mylen-\labelsep\relax,leftmargin=0pt,align=right}

\newtheorem{question}{Question}
\newtheorem{theorem}{Theorem}

\newtheorem{lemma}{Lemma}[section]
\newtheorem{prop}[lemma]{Proposition}

\theoremstyle{definition}

\theoremstyle{remark}

\numberwithin{equation}{section}

\newcommand{\Z}{\mathbb Z}
\newcommand{\N}{\mathbb N}

\newcommand{\cN}{\mathcal N}

\newcommand{\R}{\mathbb R}

\newcommand{\cI}{\mathcal I}

\newcommand{\calR}{\mathcal R}
\newcommand{\p}[1]{\medskip \noindent \emph{#1}.}
\newcommand{\definition}[1]{\medskip \noindent \emph{#1.}}
\newcommand{\cB}{\mathcal B}
\newcommand{\cC}{\mathcal C}

\newcommand{\calU}{\mathcal U}
\newcommand{\cA}{\mathcal A}
\newcommand{\bE}{\mathbb E}

\newcommand{\calT}{\mathcal{T}}

\newcommand{\Tau}{\mathcal T}

\newcommand{\vccat}{\mathfrak{VerCom}}

\newcommand{\marksur}{\mathfrak{MarkSur}}
\newcommand{\parcob}{\mathfrak{ParCob}}
\newcommand{\bn}{\noindent}

\newcommand{\tsurcat}{\mathfrak{TSurPos}}

\newcommand{\cut}{\ssearrow}

\newcommand{\calM}{\mathcal{M}}

\newcommand{\blk}{\mathfrak{b}}
\newcommand{\biL}{\mathfrak{L}}
\newcommand{\fB}{\mathfrak{B}}
\newcommand{\calV}{\mathcal{V}}
\newcommand{\calN}{\mathcal{N}}
\newcommand{\calX}{\mathcal{X}}
\newcommand{\bN}{\mathbb{N}}

\let\oldunderset\underset
\protected\def\underset{\oldunderset}

\DeclareMathOperator{\Mod}{Mod}

\DeclareMathOperator{\sep}{sep}

\DeclareMathOperator{\nosep}{nosep}

\DeclareMathOperator{\spl}{split}

\DeclareMathOperator{\Cech}{\check{C}ech}

\DeclareMathOperator{\dr}{dr}

\DeclareMathOperator{\face}{face}
\DeclareMathOperator{\adj}{adj}
\DeclareMathOperator{\tot}{tot}

\DeclareMathOperator{\lk}{lk}

\DeclareMathOperator{\Hull}{Hull}
\DeclareMathOperator{\ConvJoin}{CJoin}
\DeclareMathOperator{\HH}{\ensuremath{H}}
\newcommand{\tHH}{\widetilde{\HH}}

\DeclareMathOperator{\sus}{sus}

\begin{document}

\title{The connectivity of the complex of homologous curves}

\author{Daniel Minahan}
\address{Eckhart Hall, Chicago, IL, 60615}
\email{dminahan@uchicago.edu}


\date{\today}


\begin{abstract}
We show that the complex of homologous curves of a closed, orientable surface of genus $g$ is $(g-3)$--acyclic.  This is a key step in Putman and the author's proof that the second rational of the Torelli group is finite dimensional for surfaces of sufficiently large genus.
\end{abstract}

\vspace*{-1.5cm}

\maketitle
\thanks

\section{Introduction}\label{introsection}

Let $S_g$ be a connected, orientable, closed surface of genus $g$.  For the remainder of this paper, a \emph{curve} on $S_g$ will be a homotopy class of oriented essential embedded circles $S^1 \rightarrow S_g$.  A \emph{multicurve} $M$ will be a set of distinct curves in $S_g$ with pairwise disjoint representatives.  A curve $c \subseteq S_g$ is \emph{nonseparating} if $c$ has a nonseparating representative.  Likewise, a multicurve $M \subseteq S_g$ is \emph{nonseparating} if every curve in $M$ can be simultaneously represented by disjoint embedded circles whose union is nonseparating in $S_g$.  We let $\cC(S_g)$ denote the \emph{curve complex} of $S_g$~\cite[pg 92]{FarbMarg}.  The $k$--cells of this complex are multicurves $M \subseteq S_g$ containing $k+1$ curves.

\p{Complex of homologous curves} If $c \subseteq S_g$ is a curve, we let $[c] \in \HH_1(S_g;\Z)$ denote the homology class represented by $c$.  Let $\vec{x} \in \HH_1(S_g;\Z)$ be a \emph{primitive homology class}, i.e., a class for which if $\vec{x} = m\vec{y}$ for some $\vec{y} \in \HH_1(S_g;\Z)$ and $m \in \Z$, then $m = \pm 1$.  The \emph{complex of homologous curves}, defined by Putman~\cite{Putmantrick} and denoted $\cC_{\vec{x}}(S_g)$, is the full subcomplex of $\cC(S_g)$ generated by curves $c$ that satisfy $[c] = \vec{x}$. Putman used the work of Johnson~\cite{JohnsonII} to show that when $g \geq 3$, the complex $C_{\vec{x}}(S_g)$ is connected~\cite{Putmantrick}.  Recall that if $X$ is a path--connected, locally path--connected, semi--locally simply connected topological space, then $X$ is \emph{$n$--acyclic} if $\HH_k(X;\Z)$ vanishes for all $k \leq n$.  By convention, we say that $X$ is \emph{$(-1)$--acyclic} if $X$ is non--empty, and we say that the empty topological space is $(-2)$--acyclic.  Our first main result is the following.

\begin{theorem}\label{homolcyclethm}
Let $g \geq 2$.  Let $\vec{x} \in \HH_1(S_g;\Z)$ be a primitive homology class.  The complex $\cC_{\vec{x}}(S_g)$ is $(g-3)$--acyclic. 
\end{theorem}

\bn If $\vec{x} = 0$, then the resulting complex is called the \emph{complex of separating curves}, and is denoted $\cC_{\sep}(S_g)$.  Looijenga has proven that $\cC_{\sep}(S_g)$ is $\left(g-3\right)$--connected~\cite[Theorem 1.1]{LooijengaSeparating}.

\p{Application of Theorem~\ref{homolcyclethm} to the Torelli group}  The \emph{Torelli group}, denoted $\cI_g$, is the subgroup of the mapping class group $\Mod(S_g)$ consisting of mapping classes $\varphi$ that act trivially on $\HH_1(S_g;\Z)$.  The Torelli group $\cI_g$ acts on $\cC_{\vec{x}}(S_g)$ for any choice of primitive $\vec{x} \in \HH_1(S_g;\Z)$.  The author and Putman use Theorem \ref{homolcyclethm} in a key way to compute the second rational  homology of the Torelli group in a stable range \cite{MinahanPutmanH2}.  The complex $\cC_{\vec{x}}(S_g)$ has been used by Hatcher and Margalit to give a new proof that $\cI_g$ is generated by bounding pair maps~\cite{HatcherMargalithomologous}.  Gaster, Greene and Vlamis also related the chromatic number of $\cC_{\vec{x}}(S_g)$ to the Chillingworth homomorphism~\cite{GGV}.

\subsection{The strategy of the proof of Theorem~\ref{homolcyclethm}}

Let $g \geq 2$ and let $\vec{x} \in \HH_1(S_g;\Z)$ be a primitive homology class.  Bestvina, Bux and Margalit defined a complex called the complex of minimizing cycles, denoted $\cB_{\vec{x}}(S_g)$~\cite{BBM}.  The complex of minimizing cycles has two important properties:
\begin{itemize}
\item $\cC_{\vec{x}}(S_g)$ is a subcomplex of $\cB_{\vec{x}}(S_g)$, and
\item $\cB_{\vec{x}}(S_g)$ is contractible~\cite[Theorem E]{BBM}.
\end{itemize}
\bn  Hatcher and Margalit~\cite{HatcherMargalithomologous} prove that the relative homotopy group $\pi_1\left(\cB_{\vec{x}}(S_g), \cC_{\vec{x}}(S_g)\right)$ vanishes when $g \geq 3$.  Along with the contractibility of $\cB_{\vec{x}}(S_g)$, this implies that $\cC_{\vec{x}}(S_g)$ is connected when $g \geq 3$.  We will use their PL--Morse function to prove $\HH_k(\cB_{\vec{x}}(S_g), \cC_{\vec{x}}(S_g);\Z) = 0$ for $k \leq g - 2$.  The contractibility of $\cB_{\vec{x}}(S_g)$ then implies Theorem~\ref{homolcyclethm}.

\subsection{The complex of splitting curves} As part of the proof that $\HH_k\left(\cB_{\vec{x}}(S_g),\cC_{\vec{x}}(S_g);\Z\right) = 0$ for $k \leq g - 2$, we will also prove a result about the \emph{complex of splitting curves}.  A \emph{partitioned surface}~\cite{Putmantorelli} is a pair $\Sigma = (S,P)$ where $S$ is a connected, compact, orientable surface and $P$ is a partition of the boundary components of $S$. A \emph{block} of $P$ is an element of $P$.  The \emph{complex of separating curves} $\cC_{\sep}(\Sigma)$ is the full subcomplex of the curve complex $\cC(S)$ generated by curves $\delta$ such that each block $b \in P$ is contained entirely in one connected component of $S \cut \delta$.  Here, $\cut$ denotes Farb and Margalit's notion of cutting curves on surfaces~\cite{FarbMarg}.  The set of partitioned surfaces forms a category, where morphisms are certain inclusions that interact nicely with the partitions of the boundary components~\cite{Putmantorelli}.  Suppose now that $|P| \geq 2$, and $P$ has two distinguished blocks labeled $\fB_+$ and $\fB_-$.  The \emph{complex of splitting curves} $\cC_{\spl}(\Sigma)$ is the full subcomplex of $\cC_{\sep}(\Sigma)$ generated by curves $\delta$ such that one connected component of $S \cut \delta$ contains $\fB_+$ and the other contains $\fB_-$.  Looijenga~\cite[Theorem 1.5]{LooijengaSeparating} has shown that for such $\Sigma$, the complex $\cC_{\sep}(\Sigma)$ is $(g-2)$--connected.  We will prove in Proposition~\ref{gensplitconnlemma} that $\cC_{\spl}(\Sigma)$ is at least $(g-3 + \mathbbm{1}_{|\fB_+| \geq 2} + \mathbbm{1}_{|\fB_-| \geq 2})$--acyclic.  

\p{Comparing $\cC_{\spl}(\Sigma)$ to $\cC_{\sep}(\Sigma)$} Note that for $\Sigma = (S, \{\fB_+, \fB_-\})$ with $g(S) \geq 1$, we have $\cC_{\spl}(\Sigma) \neq \cC_{\sep}(\Sigma)$.  To see this, observe that the curve $\delta$ in Figure~\ref{vertfig} is not a vertex of $\cC_{\spl}(\Sigma)$, but is a vertex of $\cC_{\sep}(\Sigma)$.  

\begin{figure}[ht]
\begin{tikzpicture}
\node[anchor = south west, inner sep = 0] at (0,0){\includegraphics[scale=0.5]{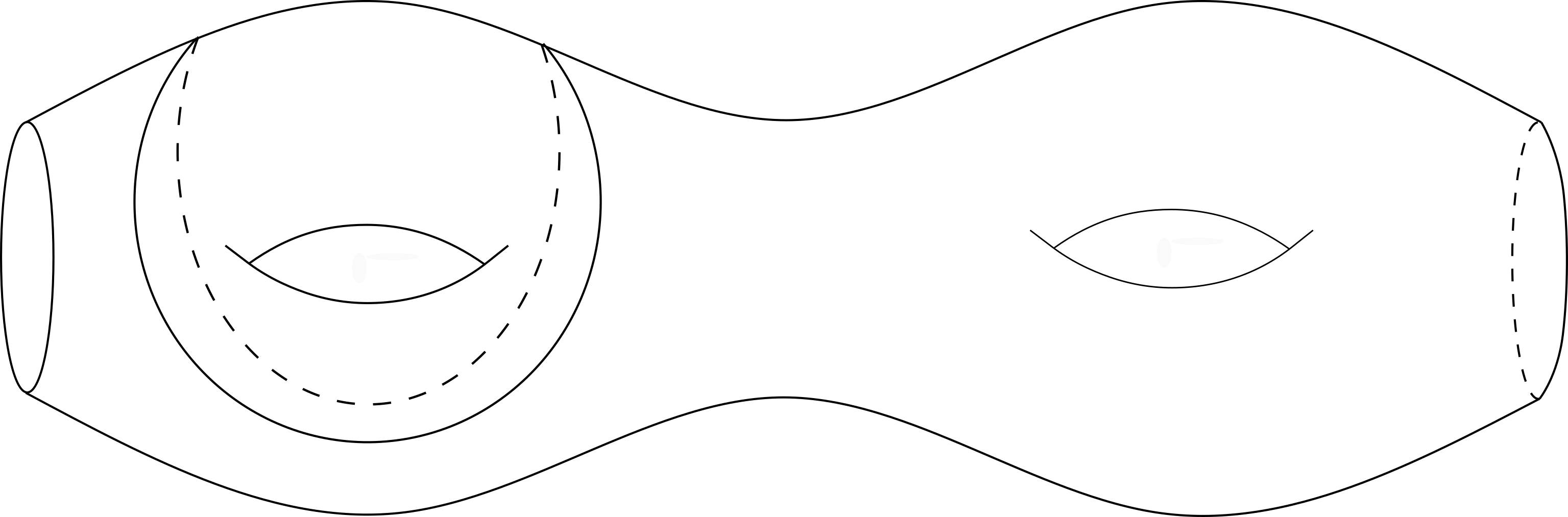}};
\node at (-0.3,1.6){\large $\fB_+$};
\node at (4.0, 1.6){\large $\delta$};
\node at (10.1, 1.6){\large $\fB_-$};
\end{tikzpicture}
\caption{The curve $\delta \in \cC_{\sep}(\Sigma)$ but not $\cC_{\spl}(\Sigma)$}\label{vertfig}
\end{figure}

\bigskip

\bn Looijenga's theorem~\cite[Theorem 1.5]{LooijengaSeparating} yields a stronger property for $\cC_{\sep}(\Sigma)$ than Proposition~\ref{gensplitconnlemma} does for $\cC_{\spl}(\Sigma)$ when either $|\fB_+|$ or $|\fB_-|$ (or both) are equal to 1.  However, if $|\fB_+|, |\fB_-| \geq 2$, Looijenga's theorem says that $\cC_{\sep}(\Sigma)$ is $(g-2)$--connected, while Proposition~\ref{gensplitconnlemma} says $\cC_{\spl}(\Sigma)$ is $(g-1)$--acyclic.  Hence if the following question could be answered, we would show that $\cC_{\sep}(\Sigma)$ is $(g-1)$--acyclic when $|\fB_+|, |\fB_-| \geq 2$.

\begin{question}
Do the relative homotopy groups $\HH_k(\cC_{\sep}(\Sigma), \cC_{\spl}(\Sigma))$ vanish for $k \leq g-1$? 
\end{question}

\bn If the answer were affirmative, then $\cC_{\sep}(\Sigma)$ would be $(g-1)$--connected via the long exact sequence in relative homotopy groups.  However, $(g-1)$ may not necessarily be the best possible bound on the connectivity of $\cC_{\sep}(\Sigma)$.  Looijenga asked the following question.

\begin{question}[{Looijenga~\cite[pg 4]{LooijengaSeparating}}]\label{Looijengaquestion}
For $\Sigma = (S,P)$ a partitioned surface, let $s(P)$ denote the number of blocks of $P$ with at least two elements.  Assume that $g(S) \geq 1$.  Is it true that $\cC_{\sep}(\Sigma)$ is $(g-4 + |P| + s(P))$--connected?
\end{question}

\bn If the answer to Question~\ref{Looijengaquestion} is affirmative, then we might expect the connectivity of $\cC_{\sep}(\Sigma)$ to be strictly higher than the connectivity of $\cC_{\spl}(\Sigma)$.  Indeed, for $\Sigma = (S,P)$ with distinguished blocks $\fB_+$, $\fB_-$ satisfying $|\fB_+|, |\fB_-| \geq 2$, we see that $g-4 + |P| + s(P) = g$, while Proposition~\ref{gensplitconnlemma} only says that $\cC_{\spl}(\Sigma)$ is $(g-1)$--connected.

\subsection{Outline of the paper}  The paper is organized into the following chunks.  
\begin{itemize}
\item General connectivity and acyclicity results (Sections~\ref{algtopsection} and~\ref{mvdksection}).
\item Proof of Proposition~\ref{gensplitconnlemma} (Section~\ref{splitsection}).
\item Proof of Theorem~\ref{homolcyclethm} (Sections~\ref{plmorseconvexsection} and~\ref{homolcurvesection}).
\end{itemize}

\bn We now give an overview of each section.

 \p{Section~\ref{algtopsection}} We discuss some general facts about connectivity and acyclicity of simplicial complexes.

\p{Section~\ref{mvdksection}} We prove Lemma~\ref{bicellularlemma}.  We begin with a packaging of some standard results about the $\Cech$--to--singular spectral sequence.  We will assume that we have some simplicial complex $A$ and a simplicial cover $\calU$ of $A$ with $\calU$ indexed in some sense by another simplicial complex $B$.  We will show that if $\tHH_k(B;\Z)$ vanishes in a range and the elements of the cover $\calU$ also satisfy some connectivity properties, then $\tHH_k(A;\Z)$ also vanishes in a range. 

\p{Section~\ref{splitsection}} We then use Proposition~\ref{bicellularlemma} to prove that the complex of splitting curves is highly acyclic. 

\p{Section~\ref{plmorseconvexsection}} We will describe how to do PL--Morse theory for cell complexes.

\p{Section~\ref{homolcurvesection}} We prove of Theorem~\ref{homolcyclethm}.  The required local connectivity properties for PL--Morse theory will be verified by inductively applying a variant of a result due to Kent--Leininger--Schleimer~\cite[Theorem 7.2]{KLS}.  The base case of this argument uses Proposition~\ref{gensplitconnlemma}.

\p{Acknowledgments}  Most of this work is part of the author's thesis.  We would like to thank our advisor Dan Margalit for many helpful math conversations and his continued support and encouragement.  We would also like Benson Farb, Eduard Looijenga, and Andy Putman for reading and providing comments on earlier drafts of this paper.  We would like to thank  Igor Belegradek, Wade Bloomquist, Katherine Williams Booth, John Etnyre, Annie Holden, Thang Le, Agniva Roy, Roberta Shapiro and Cindy Tan for helpful conversations regarding this paper.

\section{Some terminology and algebraic topology facts}\label{algtopsection}

Let $X$ be a topological space with a basepoint $x \in X$ and let $Y \subseteq X$ be a subspace with $x \in Y$.  Let $n \geq 0$ be a non--negative integer.  We say that $X$ is \emph{$n$--connected} if $\pi_k(X,x) = 0$ for every $k \leq n$.  We say that $X$ is \emph{$n$--acyclic} if $\tHH_k(X;\Z) = 0$ for every $k \leq n$.  We say that the pair $(X,Y)$ is \emph{relatively $n$--connected} if $\pi_k(X,Y) = 0$ for every $k \leq n$.  We say that the pair $(X,Y)$ is \emph{relatively $n$--acyclic} if $\HH_k(X,Y) = 0$ for every $k \leq n$.  By convention, a non--empty space is both $(-1)$--connected and $(-1)$--acyclic.  If $X$ is contractible, we will say that $X$ is $\infty$--connected.  Likewise, if $X$ has the integral homology of a point, we say that $X$ is $\infty$--acyclic.  Note also that an $n$--connected space is always $k$--connected for $k \leq n$, and similarly for acyclic.

\p{Notation}  We will let $\mathfrak{c}(X)$ and $\mathfrak{a}(X)$ denote the connectivity and acyclicity respectively of the space $X$, i.e., the maximal $n$ for which $X$ is $n$--connected (resp. $n$--acyclic).  Similarly, $\mathfrak{c}(X,Y)$ and $\mathfrak{a}(X,Y)$ will denote the connectivity and acyclicity resp. of the pair $(X,Y)$.

\p{Convention}  When we say that a space $X$ is connected, we assume that $X$ is $0$--connected, i.e., we assume that all connected spaces are nonempty.  

\medskip

\bn We require the following fact from algebraic topology (see~\cite[Lemma 2.1]{LooijengaSeparating}).

\begin{lemma}\label{joinfact}
Let $X_1,\ldots, X_n$ be a collection of topological spaces.
\begin{enumerate}[label=(\alph*)]
\item If each $X_i$ is $k_i$--connected, then the join
\begin{displaymath}
X_1 * \ldots * X_n
\end{displaymath}
\bn is $\left(-2 + \sum_{i = 1}^n \left(k_i + 2\right)\right)$--connected.
\item If each $X_i$ is $k_i$--acyclic, then the join
\begin{displaymath}
X_1 * \ldots * X_n
\end{displaymath}
\bn is $\left(-2 + \sum_{i = 1}^n \left(k_i + 2\right) \right)$--acyclic.
\end{enumerate}
\end{lemma}

\section{Covers indexed by simplicial complexes}\label{mvdksection}
The main output of this section is Lemma~\ref{bicellularlemma}, which is a result about the connectivity of complexes given as the union of certain simplicial subcomplexes.  Lemma~\ref{bicellularlemma} is a rephrasing of a result of Mirzaii and van der Kallen~\cite{MirzaiivanderKallen}.  Lemma~\ref{bicellularlemma} can also be phrased using the homotopy colimit spectral sequence~\cite{Dugger}. 

\p{Bi--cellular covers} Let $X$ and $Y$ be two simplicial complexes.  Let $\biL_X$ be a function taking a vertex $y \in Y^{(0)}$ to a subcomplex of $X$.  We will always assume that $\biL_X(y) \neq \emptyset$.  If $\sigma$ is a $k$-cell of $Y$ with vertices $y_0,\ldots, y_k$, we denote
\begin{displaymath}
\biL_X(\sigma) = \biL_X(y_0) \cap \ldots \cap \biL_X(y_k).
\end{displaymath}
\bn If $\bigcup_{y \in Y^{(0)}}\biL_X(y) = X$, we will call such an $\biL_X$ a \emph{$Y$-indexed simplicial cover of $X$}.  If $\biL_X$ is a $Y$-indexed cover of $X$, then there is an associated $X$-indexed cover of $Y$ given by setting $\biL_Y(x)$ to be the full subcomplex of $Y$ generated by vertices $y$ such that $x \in \biL_X(y)$.  We call this cover $\biL_Y$ the \emph{dual cover}.  We say that a $Y$-indexed cover of $X$ is \emph{cellularly $n$--acyclic} if for each $k$-cell $\sigma$ of $Y$, the complex $\biL_X(\sigma)$ is $(n-k)$--acyclic.  We say $\biL_X$ is \emph{bi-cellularly $n$--acyclic} if $\biL_X$ and $\biL_Y$ are both cellularly $n$--acyclic.

\begin{lemma}\label{bicellularlemma}
Let $X$ and $Y$ be finite--dimensional, countable simplicial complexes.  Let $\biL_X$ be a $Y$-indexed simplicial cover of $X$.  Suppose there is an integer $n$ such that the following hold:
\begin{itemize}
\item $\biL_X$ is bi-cellularly $n$--acyclic, and
\item $Y$ is $n$--acyclic.
\end{itemize}
\bn Then $X$ is $n$--acyclic.
\end{lemma}

\bn We defer this proof for a moment.  We will introduce the main object that we use to prove Lemma~\ref{bicellularlemma}.  

\p{Bi--cellular spectral sequence} Let $X$ and $Y$ be simplicial complexes, and let $\biL_X$ be a $Y$--indexed simplicial cover of $X$.  We have a bi-graded complex called the \emph{bi-cellular complex} given by
\begin{displaymath}
C_{p,q} = \bigoplus_{\sigma \in Y^{(p)}} C_q(\biL_X(\sigma))
\end{displaymath}
\bn where $C_*$ denotes the complex of simplicial chains on a simplicial complex.

\p{The leftward versus downward strategy} Let $C_{p,q}$ be a double complex.  There are two spectral sequences associated to $C_{p,q}$, which are the leftward and downward spectral sequences.  We denote these $\bE_{*,*}^{*, \leftarrow}$ and $\bE_{*,*}^{*, \downarrow}$.  These two spectral sequences are each constructed out of $C_{p,q}$ by two different filtrations of the total complex $C_{p+q}$.  The key point is that both $\bE_{*,*}^{*, \leftarrow}$ and $\bE_{*,*}^{*, \downarrow}$ both converge to filtrations of the total homology of $C_{p,q}$.  Denote the total homology by $\HH_{p+q}(C_{*,*})$.  Suppose that we want to compute some of the groups $\bE_{p,q}^{2,\leftarrow}$.  The strategy is as follows:
\begin{enumerate}
\item Show that $\bE_{p,q}^{*,\downarrow}$ converges to $0$ in a range $0 < p + q < n$.  This implies that $\HH_{p+q}(C_{*,*})$ converges to $0$ for $0 < p + q < n$.
\item  Use the fact that $\bE_{p,q}^{*,\leftarrow}$ must also converge to $0$ for $0 < p + q < n$ to say something about the groups $\bE_{p,q}^{2, \leftarrow}$.
\end{enumerate}

\bn This is a standard technique, used for example to show that the $G$--equivariant homology of a contractible CW--complex $X$ converges to the group homology of $G$ when $G$ acts on $X$ without rotations~\cite[Section VII]{Brownbook}.

\begin{proof}[Proof of Lemma~\ref{bicellularlemma}]
We will apply the leftward versus downward strategy discussed above.  In particular, we will show that the downward bi-cellular spectral sequence converges to $\Z$ for $p +q = 0$ and converges to 0 for $0 < p + q \leq n$, and that the leftward bi-cellular spectral sequence converges to $\HH_{p+q}(X;\Z)$ for $0 \leq p + q \leq n$.  This completes the proof since the leftward and downward sequence both converge to the total homology of $C_{*,*}$.

\p{The downward sequence} On page 1 of $\bE_{*,*}^{*, \downarrow}$, we have
\begin{displaymath}
\bE_{p,q}^{1,\downarrow} = \bigoplus_{\sigma \in Y^{(p)}} \HH_q(\biL_X(\sigma);\Z).
\end{displaymath}
\bn By hypothesis, for $0 \leq p + q \leq n$ and $q > 0$ we have
\begin{displaymath}
\bE_{p,q}^{1,\downarrow} = 0.
\end{displaymath}
\bn Then for $q = 0$ in this range we have $H_0(\biL_X(\sigma);\Z) = \Z$, so 
\begin{displaymath}
\bE_{p,0}^{1,\downarrow} = C_p(Y).
\end{displaymath}
\bn Therefore, we have $\bE_{p,0}^{2,\downarrow} = \HH_p(Y;\Z)$.  Hence $\bE_{p,q}^{*,\downarrow}$ converges to $\HH_{p+q}(Y;\Z)$, which by hypothesis is $\Z$ when $p + q = 0$ and $0$ for $0 < p + q \leq n$.

\p{The leftward sequence} By the definition of the dual cover, we have
\begin{displaymath}
C_{*,*} = \bigoplus_{\sigma \in X^{(q)}} C_p(\biL_Y(\sigma)).
\end{displaymath}
\bn Then by the same argument as the downward case, the sequence $\bE_{*,*}^{*,\leftarrow}$ converges to $\HH_{p+q}(X;\Z)$ for $p + q \leq n$.  Hence $\HH_{0}(X;\Z) \cong \Z$ and $\HH_{p+q}(X;\Z) = 0$ for $0 < p + q \leq n$, so $X$ is $n$--acyclic.
\end{proof}

\section{The Complex of splitting curves}\label{splitsection}

Our main goal in this section is to prove Proposition~\ref{gensplitconnlemma}, which says that the complex of splitting curves has the homotopy type of a wedge sum of spheres.

\definition{Homologically spherical} Let $X$ be an $n$-dimensional simplicial complex.  We say that $X$ is homologically spherical if $X$ is $(n-1)$--acyclic.  In particular, this means that $X$ has the homology of a wedge sum of $n$--spheres. 

\p{Cutting curves on surfaces} Let $S$ be a surface and $S' \subseteq S$ an isotopy class of compact submanifolds.  The notation $S \cut S'$ denotes Farb and Margalit's notion of cutting surfaces~\cite{FarbMarg}.

\definition{Partitioned surface}  Following Putman~\cite{Putmantorelli}, a \emph{partitioned surface} $\Sigma = (S,P)$ is a pair consisting of a compact, connected, oriented surface $S$ and a partition $P$ of the set of boundary components of $S$.  A \emph{block of a partition} is one set in the partition.  There is a category $\mathfrak{Sur}$ called the \emph{Torelli surface category} of partitioned surfaces.  There is a morphism $(S,P) \leq  (S',P')$ if there is an embedding $\iota:S \rightarrow S'$ satisfying the following conditions.  Let $T$ denote a connected component of $S' \cut \iota(S)$.  Let $\fB_T$ denote the boundary components of $T$ that are boundary component of $S$, and similarly $\fB'_T$. We require that the following hold:
\begin{itemize}
\item for all $T \in \pi_0(S' \cut \iota(S))$, the set $\fB_T$ is contained in some $\blk \in P$, and 
\item for each $p_1,p_2 \in \blk'$ for some $\blk' \in P'$, if $p_i \in \fB_{T_i}$ for two connected components $T_1 \neq T_2$, then there is a $\blk \in P$ with $\fB_{T_1}, \fB_{T_2} \subseteq \blk$.  
\end{itemize} We will work with a poset $\tsurcat$ consisting of objects in $\mathfrak{Sur}$.  This has the same objects as $\mathfrak{Sur}$, but we require additionally that the inclusion map $\iota: \Sigma \rightarrow \Sigma'$ is not given by filling in boundary components of $\Sigma$ with discs.  For our purposes this is a more convenient object to work with than $\mathfrak{Sur}$.  

\p{Complex of separating curves} The \emph{complex of separating curves} $\cC_{\sep}(\Sigma)$ is the full subcomplex of $\cC(S)$ generated by separating curves $\delta$ such that each block $\blk \in P$ is contained entirely in one connected component of $S \cut \delta$.  A theorem of Looijenga~\cite[Theorem 1.3]{LooijengaSeparating} tells us that this complex is $(g-2)$--connected when $\left|\pi_0(\partial S)\right| \geq 2$.

\definition{Vertex complement} Following Hatcher and Margalit~\cite{HatcherMargalithomologous}, a \emph{vertex complement} is a partitioned surface $\Sigma = (S,P)$ such that:
\begin{itemize}
\item $|P| \geq 2$, and
\item there are two distinguished blocks in $P$.  One is labeled $\fB_+$ and the other is labeled $\fB_-$.
\end{itemize}
\bn We will denote a vertex complement by
\begin{displaymath}
(S, P,\fB_+, \fB_-).
\end{displaymath}
\bn The terminology ``vertex complement" refers to $\Sigma$ being the complement of a vertex of a certain complex called the complex of minimizing cycles, which we discuss in Section~\ref{homolcurvesection}.  We will say that $\Sigma$ is a \emph{vertex complement on a surface $S$} if the underlying surface of $\Sigma$ is $S$.  The set of vertex complements is a poset which we will denote $\vccat$.  We have $\Sigma \leq \Sigma'$ if:
\begin{itemize}
\item $\Sigma \leq \Sigma'$ in $\tsurcat$, and
\item there is an inclusion $\iota:S \rightarrow S'$ realizing $\Sigma \leq \Sigma'$ in $\tsurcat$ such that there are at least two connected components of $S' \cut \iota(S)$.  The two connected components are labeled $S_+$ and $S_-$ and satisfy $\partial S_+ = \fB_+ \sqcup \fB_+'$ and $\partial S_- = \fB_- \sqcup \fB_-'$.
\end{itemize}
\bn  We will denote the genus of the underlying surface $S$ in $\Sigma = (S, P, \fB_+, \fB_-)$ by $g(\Sigma)$.

\definition{The complex of splitting curves}  Let $\Sigma = (S_g^b, P,\fB_+, \fB_-)$ be a vertex complement.  The \emph{complex of splitting cycles} $\cC_{\spl}(S,P)$ is the full subcomplex of $\cC_{\sep}(S,P)$ generated by curves $\delta$ such that $\fB_+$ and $\fB_-$ are contained in separate connected components of $S \cut \delta$.

\begin{figure}[ht]
\begin{tikzpicture}
\node[anchor = south west, inner sep = 0] at (0,0)
{\includegraphics[scale=0.5]{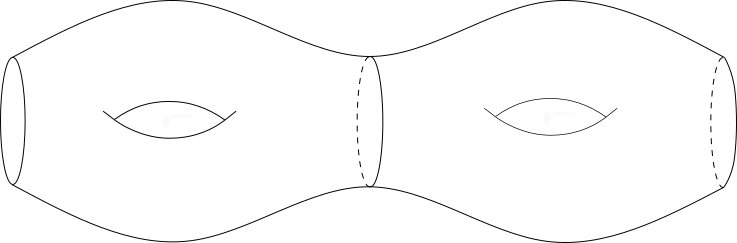}};
\node at (0.7,1.6){\large $\fB_+$};
\node at (5.3, 1.6){\large $\delta$};
\node at (10.1, 1.6){\large $\fB_-$};
\end{tikzpicture}
\caption{A splitting curve.}
\label{splitex}
\end{figure}

\medskip
\bn The remainder of this section will be devoted to the proof of the following proposition.

\begin{prop}\label{gensplitconnlemma}
Let $\Sigma = (S_g^b,P, \fB_+, \fB_-)$ be a vertex complement.  The complex $\cC_{\spl}(\Sigma)$ is homologically spherical of dimension
\begin{displaymath}
\dim(\cC_{\spl}(\Sigma)) = g -2 + \mathbbm{1}_{|\fB_+| \geq 2} + \mathbbm{1}_{|\fB_-| \geq 2}.
\end{displaymath}
\end{prop} \bn Here, $\mathbbm{1}_{\mathfrak{S}}$ denotes the indicator function, i.e, the function whose value is $1$ if statement $\mathfrak{S}$ is true, and $0$ otherwise.  

\subsection{The nonseparating arc complex}\label{noseparcsection}

Let $S = S_g^b$ be a surface and $\Sigma = (S,P, \fB_+, \fB_-)$ be a vertex complement.  Let $Q$ be a set of points on boundary components of $\fB_+$ such that $\left|Q \right| \geq 2$ and such that each element of $\fB_+$ contains at least one point in $Q$.  Such a $Q$ will be called an \emph{admissible marking}.  The \emph{arc complex} (see, e.g., Schleimer \cite{Schleimernotes}) $\cA(\Sigma, Q)$ is a complex defined as follows:
\begin{itemize}
\item a vertex of $\cA(\Sigma, Q)$ is an isotopy class of arcs with endpoints distinct points of $Q$; and
\item a $k$--cell is $k+1$ vertices which can be represented by $k+1$ arcs that only intersect possibly at their endpoints.
\end{itemize}  It is known that $\cA(\Sigma,Q)$ is contractible \cite{Hatcher}.  Likewise, the \emph{nonseparating arc complex} $\cA_{\nosep}(\Sigma, Q)$ is defined as above, except we additionally require that if $\sigma = \{\alpha_0,\ldots, \alpha_k\}$ is a $k$--cell, then $S \cut \sigma$ is connected.  Harer proves the following.

\begin{lemma}[{{\cite[Theorem 1.3]{Harerstability}}}]\label{harerlemma}
Let $\Sigma$ be a vertex complement and $Q$ an admissable marking.  The complex $\cA_{\nosep}(\Sigma, Q)$ is $(2g - 3 + \left|\fB_+\right|)$--connected.
\end{lemma}

\subsection{The outline of the of Proposition~\ref{gensplitconnlemma}}\label{splitsubsection}
We now show that $\cC_{\spl}(\Sigma)$ is homotopically spherical.  We will begin by describing the outline of the proof of Proposition~\ref{gensplitconnlemma}.  We will then carry out some of the steps in the outline in a pair of lemmas.  We will conclude Section~\ref{splitsubsection}, and Section~\ref{splitsection}, by proving Proposition~\ref{gensplitconnlemma}.

\p{The setup of the proof of Proposition~\ref{gensplitconnlemma}}  The proof follows by induction on the poset $\vccat$.  Let $\Sigma = (S, P, \fB_+, \fB_-)$ be a vertex complement such that Proposition~\ref{gensplitconnlemma} holds for all $\Tau < \Sigma$.  Suppose without loss of generality that $|\fB_+| \geq |\fB_-|$.  Choose a set $Q$ of compact subintervals of $\partial \Sigma$ with $|Q|$ minimal such that $(S,Q) \in \marksur$.  Let $\cA = \cA_{\nosep}(\Sigma,Q)$.  If $\sigma$ is a cell of $\cC_{\spl}(\Sigma)$, let $\biL_{\cA}(\sigma)$ be the full subcomplex of  $\cA$ generated by arcs disjoint from $\sigma$.  Similarly, for cells $\tau \in \cA$, let $\biL_{\cC}(\tau)$ denote the subcomplex of $\cC_{\spl}(\Sigma)$ generated by cells $\sigma$ with $\sigma$ disjoint from $\tau$.  Let $C_{p,q}$ be  the bi-cellular complex
\begin{displaymath}
\bigoplus_{\sigma \in \cA^{(p)}} C_q(\biL_{C}(\sigma)).
\end{displaymath}
\bn Let $\bE_{*,*}^{*, \downarrow}$ and $\bE_{*,*}^{*, \leftarrow}$ be the downward and leftward bi-cellular spectral sequences respectively.  The goal is to prove the following two facts:
\begin{enumerate}
\item the downward sequence converges to $\Z$ for $p + q = 0$ and converges to $0$ for $0 < p + q < \dim(\cC_{\spl}(\Sigma))$, and 
\item the leftward sequence converges to $\HH_{p+q}(\cC_{\spl}(\Sigma))$ for $0 \leq p + q< \dim(\cC_{\spl}(\Sigma))$.  
\end{enumerate}
\bn Given these two facts, $\tHH_{p+q}(\cC_{\spl}(\Sigma)) = 0$ for $0 \leq p + q < \dim(C_{\spl}(\Sigma))$.  We will then prove simple connectivity separately.  

\subsection{Convergence of the leftward sequence} Fact (1) follows from essentially the same argument as in the proof of Lemma~\ref{bicellularlemma} so the bulk of Section~\ref{splitsubsection} will be devoted to proving Fact (2), which is recorded as the following result.

\begin{lemma}\label{downconvlemma}
Let $\Sigma \in \vccat$ such that for every $\Tau < \Sigma$, the complex $\cC_{\spl}(\Tau)$ is homologically spherical.  Let $\bE_{p,q}^{1, \leftarrow}$ be the leftward spectral sequence discussed above.  Then
\begin{displaymath}
\bE_{p,q}^{1,\leftarrow} \Rightarrow \HH_{p +q}(\cC_{\spl}(\Sigma);\Z)
\end{displaymath} 
\bn for $p+q < \dim(\cC_{\spl}(\Sigma))$.
\end{lemma}

\bn The idea of the proof is to use a variant of PL--Morse theory for homology in coefficient systems.  In particular, we will prove the following lemma.

\begin{lemma}\label{splspecdeclink} Let $\Sigma \in \vccat$ be as in Lemma~\ref{downconvlemma}.  Let $p,q \geq 0$ be integers with $p > 0$ and $p + q < \dim(\cC_{\spl}(\Sigma))$.  Let $\delta \in \cC_{\spl}(\Sigma)$.  Label the connected components of $\Sigma \cut \delta$ by $\Sigma_+$ and $\Sigma_-$ where $\fB_+ \subseteq \Sigma_+$ and $\fB_- \subseteq \Sigma_-$.  Suppose that $g(\Sigma_+) + |\fB_+| -3 < p$.  Then $\cC_{\spl}(\Sigma_-)$ is at least $(q-1)$--acyclic. 
\end{lemma} 

\bn Note that the assumption on $\delta$ in Lemma~\ref{splspecdeclink} says that if $(\Sigma, Q) \in \marksur$, then Lemma \ref{harerlemma} does not necessarily tell us that the homology group $\HH_p(\biL_{\cA}(\delta);\Z)$ vanishes.  The content of Lemma~\ref{splspecdeclink} is that for these $\delta$, the acyclicity of $C_{\spl}(\Sigma_-)$ is high enough to allow us to carry out something resembling PL--Morse theory inside the spectral sequence $\bE_{p,q}^{1, \leftarrow}$.

\begin{proof}[Proof of Lemma~\ref{splspecdeclink}] Let $p' = g(\Sigma_+) + |\fB_+| - 3$.  We have assumed that $p' < p$, so we have
\begin{displaymath}
\dim(\cC_{\spl}(\Sigma)) - p' \geq q + 2,
\end{displaymath}
\bn since $p + q < \dim(\cC_{\spl}(\Sigma))$.  Now, we have 
\begin{displaymath}
\dim(\cC_{\spl}(\Sigma)) = g(\Sigma) -2 + \mathbbm{1}_{|\fB_-| \geq 2} + \mathbbm{1}_{|\fB_+| \geq 2}.
\end{displaymath}
\bn We also have $g(\Sigma_+) + g(\Sigma_-) = g(\Sigma)$.  Therefore, we have
\begin{align*}
\dim(\cC_{\spl}(\Sigma_-)) &= g(\Sigma_-) -2 + \mathbbm{1}_{|\fB_-| \geq 2}\\
&= g(\Sigma) - g(\Sigma_+) + \mathbbm{1}_{|\fB_-| \geq 2} -2.
\end{align*} We now substitute in $-g(\Sigma_+) = |B_+| - 3 - p' \geq |B_+| -3 - p$ to see that
\begin{align*}
\dim(\cC_{\spl}(\Sigma_-)) &\geq g(\Sigma) - p -5 +|\fB_+| + \mathbbm{1}_{|B-| \geq 2}\\
&= \dim(\cC_{\spl}(\Sigma)) - \mathbbm{1}_{|\fB_+| \geq 2} - p' - 3 + |\fB_+|\\
&\geq q -1 +|\fB_+| - \mathbbm{1}_{|\fB_+| \geq 2}\\
&\geq q .
\end{align*}
\bn Then $\Sigma_- < \Sigma$, so by hypothesis we have $\mathfrak{a}(\cC_{\spl}(\Sigma_-)) \geq q-1$.
\end{proof}

\bn We are now ready to prove Lemma~\ref{downconvlemma}. 

\begin{proof}[Proof of Lemma~\ref{downconvlemma}]
On page 1, the sequence $\bE_{p,q}^{1,\leftarrow}$ is given by
\begin{displaymath}
\bE_{p,q}^{1,\leftarrow} = \bigoplus_{\tau \in \cC_{\spl}(\Sigma)^{(q)}}\HH_p(\biL_{\cA}(\tau)).
\end{displaymath}
\bn Observe that in the column $p = 0$, the chain complex $\bE_{0,q}^{1,\leftarrow}$ is identified with $C_*(\cC_{\spl}(\Sigma))$, since Lemma \ref{harerlemma} implies that each cell $\tau \in \cC_{\spl}(\Sigma)$ has $\biL_{\cA}(\tau)$ connected except in the case that the connected component $\Sigma_+$ of $\Sigma \cut \tau$ containing $B_+$ has $|\fB_+| = 2$ and $g(\Sigma_+) = 0$.  For each vertex $\delta \in \cC_{\spl}(\Sigma)$ with $g(\Sigma_+) = 0$, choose an arc $\alpha_{\delta} \in \biL_{\cA}(\delta)$.  Let $K$ denote the set of such $\delta$, and observe that no cell in $\cC_{\spl}(\Sigma)$ contains multiple vertices in $K$.  Therefore, the choice of $\alpha$ for each $\delta \in K$ induces a section $\rho$ of the natural map $\bE_{0,q}^{1, \leftarrow} \rightarrow C_*(\cC_{\spl}(\Sigma);\Z)$ as follows. If $\tau \in \cC_{\spl}(\Sigma)$ is a cell with $\tau^0 \cap K = \emptyset$, $\rho$ sends a generator of $\Z[\tau]$ to a generator of $\HH_0(\biL_{\cA}(\tau);\Z)$.  Otherwise, if $\delta \in \tau^0 \cap K$, $\rho$ sends a generator of $\Z[\tau]$ to the class $[\alpha] \in \HH_0(\biL_{\cA}(\tau);\Z)$.  Hence it is enough to show that $\bE_{p,q}^{2,\leftarrow} = 0$ for $p > 0$ and $0 < p + q < \dim(\cC_{\spl}(\Sigma))$. 
 
Pick a pair $p,q$ with $p > 0$ and $p + q < \dim(\cC_{\spl}(\Sigma))$.  We will show that $\bE_{p,q}^{2,\leftarrow} = 0$.  For a vertex $\delta \in \cC_{\spl}(\Sigma)$, let $W(\delta) = g(\Sigma_+) +|\fB_+| - 3$, where $\Sigma_+$ is the connected component of $\Sigma \cut \gamma$ that contains $\fB_+$.  Let $\Sigma_-$ be the connected component of $\Sigma \cut \delta$ containing $\fB_-$,.  An example of $\delta$, $\Sigma_+$ and $\Sigma_-$ can be seen in Figure~\ref{splitdecex}.

\begin{figure}[ht]
\begin{tikzpicture}
\node[anchor=south west, inner sep = 0] at (0,0){\includegraphics[scale = 0.7]{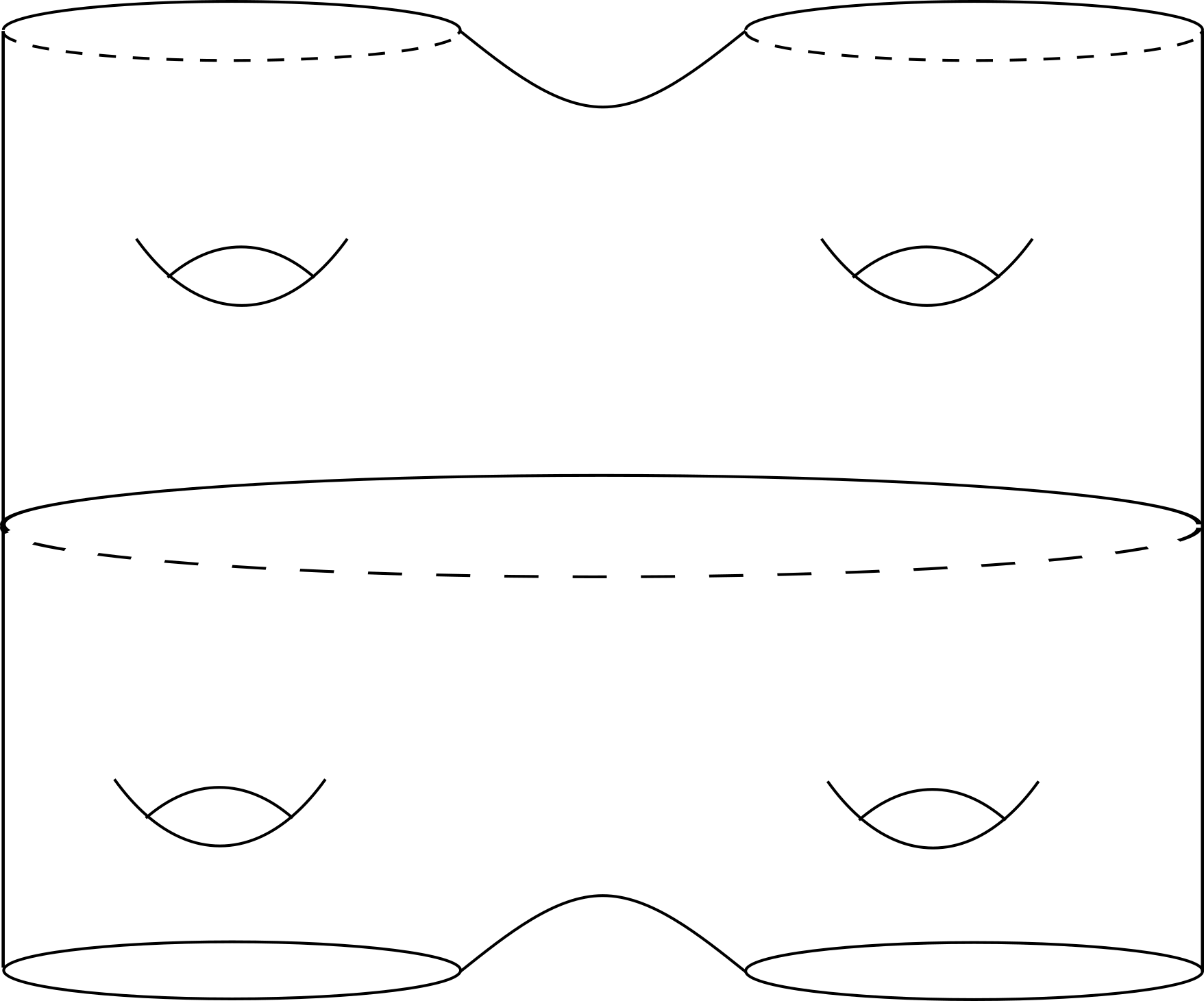}};
\node at (4.0,3.7){\large $\delta$};
\node at (4.0,5){\large $\Sigma_{-}$};
\node at (4.0, 2){\large $\Sigma_+$};
\node at (4.0, -0.6){\large $\fB_+$};
\node at (4.0, 6.8){\large $\fB_-$};
\end{tikzpicture}
\caption{A curve $\delta$ with $W(\delta) = 1$.}\label{splitdecex}
\end{figure}

\bn Let $\bE_{p,*}^{1, \leftarrow}$ denote the chain complex in the $p$th column of the spectral sequence $\bE_{*,*}^{1, \leftarrow}$.  We will show that $\bE_{p,q}^{2, \leftarrow} = \HH_q\left(\bE_{p,*}^{1, \leftarrow};\Z\right) = 0$.  There is a filtration $X_k$ of $\cC_{\spl}(\Sigma)$ given by $X_k =  \bigcup_{\tau \in F_k} \tau $, where
\begin{displaymath}
F_k = \left\{\tau \subseteq \cC_{\spl}(\Sigma): \min_{\delta \in \tau^{(0)}} W(\delta) \geq p - k\right\}
\end{displaymath}
\bn This filtration starts at $k = 0$ and runs to $k = p$.  Each $X_k$ induces a subcomplex of $\bE_{p,*}^{1,\leftarrow}$ by
\begin{displaymath}
\bE_{p,q}^{1, \leftarrow}(X_k) = \bigoplus_{\tau \in \cC_{\spl}(\Sigma)^{(q)}: \tau \subseteq X_k} \HH_p(\biL_{\cA}(\tau)).
\end{displaymath}
\bn  By construction, these $\bE_{p,*}^{1, \leftarrow}(X_k)$ are a filtration of $\bE_{p,*}^{1, \leftarrow}$.  It therefore suffices to prove the following two facts:
\begin{itemize}
\item $\HH_q(\bE_{p,*}^{1, \leftarrow}(X_0)) = 0$, and
\item $\HH_q(\bE_{p,*}^{1, \leftarrow}(X_k), \bE_{p,*}^{1, \leftarrow}(X_{k-1})) = 0$ for $1 \leq k \leq p$.
\end{itemize}
\bn We will prove each of these in turn.

\p{$\HH_q(\bE_{p,*}^{1, \leftarrow}(X_0);\Z) = 0$}  Each $\tau \in X_0$ has $\mathfrak{a}(\biL_{\cA}(\tau)) \geq p$ by Lemma \ref{harerlemma}.  Hence $\bE_{p,*}^{1, \leftarrow}(X_0)$ is identically $0$.  

\p{$\HH_q(\bE_{p,*}^{1, \leftarrow}(X_k), \bE_{p,*}^{1, \leftarrow}(X_{k-1});\Z) = 0$ for $1 \leq k \leq p$}  Let 
\begin{displaymath}
\Delta_k = \{\delta \in \cC_{\spl}(\Sigma): W(\delta) = p - k\}.
\end{displaymath}
\bn Note that $\Delta_k$ is a discrete set, i.e., no two vertices in $\Delta_k$ are adjacent.  Assume that $\Delta_k$ is ordered arbitrarily, which induces a filtration $Y_k^j$ of $X_k$, where $Y_k^j$ is given by attaching to $X_{k-1}$ the first $j$ elements in the order on $\Delta_k$.  Each $Y_k^j$ induces a subcomplex of $\bE_{p,*}^{1, \leftarrow}(X_k)$, which we denote $\bE_{p,*}^{1, \leftarrow}(Y_k^j)$.  We have the following claim.

\p{Claim} For any $j \geq 1$, the relative homology $\HH_q(\bE_{p,*}^{1, \leftarrow}(Y_k^j), \bE_{p,*}^{1, \leftarrow}(Y_k^{j-1})) = 0$.

\bigskip

\bn Given the claim, we have $\HH_q(\bE_{p,*}^{1, \leftarrow}(X_k), \bE_{p,*}^{1, \leftarrow}(X_{k-1})) = 0$ for $1 \leq k \leq p$.  Since $\HH_q(\bE_{p,*}^{1, \leftarrow}(X_0))= 0$, we have $\HH_q(\bE_{p,*}^{1, \leftarrow}(X_p)) = \HH_q(\bE_{p,*}^{1, \leftarrow}) = 0$, as desired.

\p{Proof of claim}  Let $\delta$ be the unique vertex in $Y_k^j \setminus Y_k^{j-1}$.  If $\delta \in \cC_{\spl}(\Sigma)$ is a vertex, let $a_W(\delta)$ denote the subcomplex of $\cC_{\spl}(\Sigma)$ consisting of all $\delta' \in \lk(\delta)$ with $W(\delta') > W(\delta)$.  Note that if $\tau \subseteq a_W(\delta)$ is a cell, the inclusion map $\biL_{\cA}(\tau * \delta) \hookrightarrow \biL_{\cA}(\delta)$ is an isomorphism.  Hence, there is short exact sequence of chain complexes
\begin{displaymath}
0 \rightarrow C_*(a_W(\delta)) \otimes \HH_p(\biL_{\cA}(\delta);\Z) \rightarrow \bE_{p,*}^{1, \leftarrow}(Y_k^{j-1}) \oplus C_*\left(\delta *a_W(\delta)\right) \otimes \HH_p(\biL_{\cA}(\delta);\Z) \rightarrow \bE_{p,*}^{1, \leftarrow}(Y_k^{j})\rightarrow 0.
\end{displaymath}
\bn Now, by Lemma~\ref{splspecdeclink} and the assumption that Lemma \ref{gensplitconnlemma} holds for all $\Tau \subsetneq \Sigma$, we have $\tHH_{k}(a_W(\delta)) = 0$ for $k \leq q -1$, and thus we have
\begin{itemize}
\item $\HH_{k}(a_W(\delta) \otimes \HH_p(\biL_{\cA}(\delta);\Z)) = 0$ for $1 \leq k \leq q -1$, and 
\item $\HH_{k}(a_W(\delta) \otimes \HH_p(\biL_{\cA}(\delta);\Z)) =  \HH_p(\biL_{\cA}(\delta);\Z)$.
\end{itemize}
\bn  The complex $\delta * a_W(\delta)$ is contractible, so we have
\begin{itemize}
\item $\HH_{k}(\left( \delta * a_W(\delta)\right) \otimes \HH_p(\biL_{\cA}(\delta);\Z) = 0$ for $1 \leq k $, and 
\item $\HH_{0}(\left( \delta * a_W(\delta)\right) \otimes \HH_p(\biL_{\cA}(\delta);\Z) = \Z$.
\end{itemize}
\bn Therefore the  pushforward map
\begin{displaymath}
\HH_k(\bE_{p,*}^{1, \leftarrow}(Y_k^j)) \rightarrow \HH_k(\bE_{p,*}^{1, \leftarrow}(Y_k^{j-1});\Z))
\end{displaymath}
\bn is surjective for $k = q$ and an isomorphism for $k  \leq q-1$, so the claim holds.
\end{proof}

\subsection{Homologically spherical} We are now ready to show that $\cC_{\spl}(\Sigma)$ is homologically spherical.  
\begin{proof}[Proof of Proposition \ref{gensplitconnlemma}]
We will induct on the poset $\vccat$.  Namely, if $\Sigma$ is a vertex complement, we will assume that Proposition~\ref{gensplitconnlemma} holds for all $\Tau \in \vccat$ with $\Tau < \Sigma$, and show that it holds for $\Sigma$ as well.

\p{Base cases}  Our base cases are given by any choice of $g$, $|\fB_+|$ and $|\fB_-|$ such that
\begin{displaymath}
g - 2+ \mathbbm{1}_{|\fB_+|\geq 2} + \mathbbm{1}_{|\fB_-| \geq 2} \geq 0.
\end{displaymath} 
\bn This is equivalent to
\begin{displaymath}
g + \mathbbm{1}_{|\fB_+|\geq 2} + \mathbbm{1}_{|\fB_-| \geq 2} \geq 2,
\end{displaymath}
and all such choices of $g, |\fB_+|, |\fB_-|$ are as follows:
\begin{itemize}
\item $g \geq 2$,
\item $g \geq 1, |\fB_+| \geq 2$ or $|\fB_-| \geq 2$, and
\item $|\fB_+|, |\fB_-| \geq 2$.
\end{itemize}
\bn  In all these cases $C_{\spl}(\Sigma)$ is non--empty, so the result holds.

\p{Induction on $\vccat$}  Let $\Sigma = (S, P, \fB_+, \fB_-)$ be a vertex complement such that the proposition holds for all $\Tau < \Sigma$.  Suppose without loss of generality that $|\fB_+| \geq |\fB_-|$.  Let $C_{p,q}$, $\bE_{p,q}^{*, \leftarrow}$ and $\bE_{p,q}^{*, \downarrow}$ be as discussed in the beginning of Section~\ref{splitsubsection}.  It suffices to prove that the downward sequence converges to $\Z$ for $p + q = 0$ and converges to $0$ for $0 < p + q < \dim(\cC_{\spl}(\Sigma))$, and that the leftward sequence converges to $\HH_{p+q}(\cC_{\spl}(\Sigma))$ for $0 < p + q< \dim(\cC_{\spl}(\Sigma))$. 

\p{The downward spectral sequence converges to $0$ for $0 < p + q < \dim(\cC_{\spl}(\Sigma))$}  On page 1, the downward spectral sequence $\bE_{*,*}^{*,\downarrow}$ is given by
\begin{displaymath}
\bE_{p,q}^{1, \downarrow} = \bigoplus_{\sigma \in \cA^{(p)}}\HH_q(\biL_C(\sigma)).
\end{displaymath}
\bn If $|\fB_+| = 2$ and $\sigma$ is a vertex, then $\biL_{\cC}(\sigma)$ is contractible.  For all other situations, the inductive hypothesis says that the groups $\tHH_q(\biL_{\cC}(\sigma))$ are trivial for $\sigma$ a $p$-cell and $0 \leq q < \dim(\cC_{\spl}(\Sigma)) - p$.  Therefore we have
\begin{displaymath}
\bE_{p,q}^{*, \downarrow} \Rightarrow \HH_{p + q}(\cA_{\nosep}(\Sigma, Q))
\end{displaymath}
\bn for $ p + q < \dim(\cC_{\spl}(\Sigma))$.  But $g  - 3 + |\fB_+| \geq \dim(\cC_{\spl}(\Sigma)) - 1$, so by Lemma \ref{harerlemma}, $\bE_{p,q}^{*,\downarrow} \Rightarrow 0$ for $0 < p + q < \dim(\cC_{\spl}(\Sigma))$ and $\Z$ for $p + q= 0$.  

\p{The leftward spectral sequence converges to $\HH_{p+q}(\cC_{\spl}(\Sigma))$}  This is the content of Lemma~\ref{downconvlemma}.

\bn Since both $\bE_{p,q}^{1,\leftarrow}$ and $\bE_{p,q}^{1,\downarrow}$ converge to the total homology of $C_{p,q}$, we have $\tHH_{p+q}(\cC_{\spl}(\Sigma)) = 0$ for $p + q < \dim(\cC_{\spl}(\Sigma))$, as desired.
\end{proof} 

\section{PL--Morse theory for cell complexes}\label{plmorseconvexsection}

Recall from the introduction that the proof of Theorem~\ref{homolcyclethm} proceeds by using PL--Morse theory on the complex of minimizing cycles $\cB_{\vec{x}}(S_g)$~\cite{BBM}.  A reference for PL--Morse theory in general can be found in the work of Bestvina \cite{BestvinaMorse}. The complex $\cB_{\vec{x}}(S_g)$ is not a simplicial complex, so we must explain how to use PL--Morse theory on certain types of non--simplicial complexes.  The main goal is to prove Lemma~\ref{plconvexapp}, which describes a method for proving that a complex is highly acyclic.

\p{Convex hulls}  Let $\R^{\N} = \bigoplus_{n \in \N} \R$.  Let $S_1, S_2 \subseteq \R^{\N}$ be two subsets.  The convex hull of $S_1$, denoted $\Hull(S_1)$, is the set
\begin{displaymath}
\{t_0s_0 + \ldots + t_ns_n: n \in \Z_{\geq 0}, s_0,\ldots, s_n \in S, \sum_{i=0}^n t_i = 1, t_i \geq 0 \text{ for all } 0 \leq i \leq n\}.
\end{displaymath}
\bn  The convex join of $S_1$ and $S_2$, denoted $\ConvJoin(S_1,S_2)$, is the set
\begin{displaymath}
\{t_1s_1 + t_2s_2: s_i \in S_i, t_1 + t_2 = 1, t_i \geq 0\}.
\end{displaymath}

\p{Locally linear cell complex} A finite dimensional cell complex $X$ is \emph{locally linear} if there is an inclusion $\iota: X \hookrightarrow \R^{\N}$ such that each cell $\sigma$ of $X$ is the convex hull of its vertices and $\partial \sigma$ is the union of the faces of $\iota(\sigma)$.  We will conflate $X$ with its image under this map.  Let $W:X^{(0)} \rightarrow \N$ be a PL--Morse funtion.  We say that $W$ is a \emph{linear PL--Morse function} if $W$ is the restriction of some linear function $\R^{\N} \rightarrow \R$, which by abuse of notation we will refer to as $W$ as well.  Linear PL--Morse functions have the following property.

\begin{lemma}\label{convexfacefact}
Let $\sigma$ be a cell of a locally linear cell complex $X$ and let $W$ be a linear PL--Morse function on $X$.  The set of all points $x \in \sigma$ with $W(x)$ of maximal weight among the points of $\sigma$ is a face of $\sigma$, and is the convex hull of the vertices of $\sigma$ of maximal weight.
\end{lemma}

\begin{proof}
Let $M \subseteq \sigma$ denote the set of points of maximal weight.  By definition this is a face of $\sigma$ since $M$ is the set of points maximizing the linear function $W$.  Then $M$ must be the convex hull of its vertices as well, so the result follows.
\end{proof}

\p{Sharp PL--Morse functions} Let $X$ be a locally linear cell complex.  We say that a function $W:X^{(0)} \rightarrow \bN$ is \emph{sharp} if, for every cell $\rho \subseteq X$ where $W|_{\rho^0} \neq 0$, the function $W|_{\rho^0}$ has a unique maximum.

\p{Quasi--linear PL--Morse functions} We say that a function $W:X^0\rightarrow \N$ is \emph{quasi--linear} if $W$ is either linear or sharp, and we will refer to these as PL--Morse functions.  We say that a cell $\sigma \subseteq X$ is \emph{$W$--constant} if the following hold:
\begin{itemize}
\item $W(v) = W(w)$ for every $v,w \in \sigma^{(0)}$; and
\item $W(v) > 0$ for all $v \in \sigma^0$.
\end{itemize}

\p{Descending links in locally linear cell complexes}  Let $X$ be a locally linear cell complex and let $W$ be a quasi--linear PL--Morse function on $X$.  Let $\sigma \subseteq X$ be a $W$--constant cell.  We will define three different versions of the descending link $d_W(\sigma)$, and then show that all are homotopy equivalent.  Let $\calR_{\sigma}$ denote the set of all cells $\rho \subseteq X$ such that
\begin{itemize}
\item $\sigma$ is a face of $\rho$, and
\item $W(r) \leq W(\sigma)$ for all $r \in \rho^{(0)}$, with equality if and only if $r \in \sigma$.  
\end{itemize}
\bn The variants of the descending link are as follows.

\begin{itemize}
\item \emph{The facial descending link $d_W^{\face}(\sigma)$.} Let $\rho \in \calR_\sigma$ be a cell and let $\calT(\rho)$ be the set of all faces $\tau \subseteq \rho$ such that $\tau \cap \sigma = \emptyset$.  We define
\begin{displaymath}
d_W^{\face}(\sigma) = \bigcup_{\rho \in \calR_\sigma} \calT(\rho).
\end{displaymath}
\item \emph{The adjacent descending link $d_W^{\adj}(\sigma)$.}  For all $\rho \in \calR_\sigma$, define $V(\rho)$ to be the set of all vertices $v \in \rho^{(0)}$ with $v \not \in \sigma$ such that $v*w$ is an edge of $X$ for some $w \in \sigma$.  We define
\begin{displaymath}
d_{W}^{\adj}(\sigma) = \bigcup_{\rho \in \calR_\sigma} \Hull(V(\rho)).
\end{displaymath}
\item \emph{The total descending link $d_W^{\tot}(\sigma)$.} This is given by
\begin{displaymath}
d_{W}^{\tot}(\sigma) = \ConvJoin(d_W^{\face}(\sigma), d_W^{\adj}(\sigma)).
\end{displaymath}
\end{itemize}
\bn Note that there are two inclusions
\begin{displaymath}
d_{W}^{\face}(\sigma) \hookrightarrow d_W^{\tot}(\sigma) \hookleftarrow d_W^{\adj}(\sigma).
\end{displaymath}
\bn We first prove the following.
\begin{lemma}\label{dlcontractlemma}
Let $W$ be a quasi--linear PL--Morse function on a locally linear cell complex $X$.  Let $\sigma \subseteq X$ be a $W$--constant cell.  Let $\rho \in \calR_{\sigma}$ be a cell.  Then $d_W^{\dagger}(\sigma) \cap \rho$ is contractible for any choice of $\dagger = \face, \tot, \adj$.  
\end{lemma}

\begin{proof}
We prove this for each choice of $\dagger$ in turn.

\p{$\dagger = \adj$}  We will show that $d_W^{\adj}(\sigma) \cap \rho = \Hull(V(\rho))$, where $V(\rho)$ is as in the definition of $d_W^{\face}$.  
The only potentially non--obvious point here is that if $\rho$ is a face of another cell $\rho' \subseteq \calR_{\sigma}$, then it may be the case that $\Hull(V(\rho')) \cap \rho \supsetneq \Hull(V(\rho))$.  Suppose in fact that this does happen, so there is a point $x \in \Hull(V(\rho')) \cap \rho \cut \Hull(V(\rho))$.  Then $x$ is on a line connecting two points $t_1,t_2 \in V(\rho')$.  These $t_1$ and $t_2$ cannot lie $\rho$, since otherwise $x \in \Hull(V(\rho))$.  But no such $t_1$ and $t_2$ can exist since $\rho$ is a face of $\rho'$, so we have $\Hull(V(\rho)) = d_W^{\adj}(\sigma) \cap \rho$.  

\p{$\dagger = \tot$} This follows by a similar line of reasoning to $\dagger = \adj$.

\p{$\dagger = \face$}  As for $\adj$, define $\calT(\rho)$ as in the definition of $d_W^{\face}(\sigma)$.  Then $\calT(\rho)$ is given by the union of faces of $\rho$.  If $x \in \rho$ were a point in the interior of $\rho$ that was contained in $\calT(\rho')$ for some other cell $\rho'$, then $\rho$ would be a face of $\rho'$.  But then $\rho \subseteq \calT(\rho')$ implies that $\sigma \cap \rho = \emptyset$ by definition of $\calT$, which contradicts the definition of $\calR_{\rho}$.  
\end{proof}

\bn We now have the following lemma that describes the relationships between the three types of descending links.
\begin{lemma}\label{dlvarlemma}
Let $W$ be a quasi--linear PL--Morse function on a locally linear cell complex $X$.  Let $\sigma \subseteq X$ be a $W$--constant cell.  Then each inclusion
\begin{displaymath}
d_{W}^{\face}(\sigma) \hookrightarrow d_W^{\tot}(\sigma) \hookleftarrow d_W^{\adj}(\sigma)
\end{displaymath}
\bn is a weak homotopy equivalence.
\end{lemma}
\begin{proof}
Let $\calR_\sigma$ be as above.  For each choice of $\dagger = \face, \tot, \adj$ and for each $\rho \in \calR^{\max}$ let
\begin{displaymath}
U_{\rho}^\dagger = \rho \cap d_W^\dagger(\sigma).
\end{displaymath}
Let $\calR^{\max}_\sigma$ denote the set of maximal cells in $\calR_\sigma$, i.e., the set of all $\rho$ such that there is no $\rho' \in \calR_{\sigma}$ with $\rho \subsetneq \rho'$.  Let $\cN(U_{\rho}^{\dagger})$ denote the nerve of the set $\{U_{\rho}^{\dagger}\}_{\rho \in \calR_{\sigma}^{\max}}$.  Each of these spaces and their non--empty intersections are contractible by Lemma~\ref{dlcontractlemma}. The nerve lemma implies that we have isomorphisms of homology groups
\begin{displaymath}
\HH_*(\cN(U_{\rho}^{\dagger})) \cong \HH_*(d_W^{\dagger}(\sigma)).
\end{displaymath}
\bn It now remains to show that for a choice of $\dagger = \adj, \face$, the induced map of chain complexes
\begin{displaymath}
C_*(\cN(U_{\rho}^{\dagger});\Z) \rightarrow C_*(\cN(U_{\rho}^{\tot});\Z)
\end{displaymath}
\bn is a homotopy equivalence.  In fact, we will show that the above map induces an isomorphism of chain complexes.  It is clear that the induced map is a well--defined injection for either choice of $\dagger \in \{\face, \adj\}$, so it remains to prove that this induced map is a surjection.

 Let $\rho_0,\ldots,\rho_n \in \calR_{\sigma}^{\max}$ be a collection of cells with $U_{\rho_0}^{\tot} \cap \ldots \cap U_{\rho_n}^{\tot} \neq \emptyset$.  Let $\rho = \rho_0 \cap \ldots \cap \rho_n$, which is necessarily a cell of $X$.  Since $\rho$ is a cell of $X$ and $\sigma \subsetneq \rho$, there is a cell $\tau \subseteq \rho$ with $\tau \cap \sigma =\emptyset$ and $\tau \subseteq U_{\rho}^{\tot}$.  We can choose $\tau$ to contain a vertex adjacent to $\sigma$, since $\rho$ is a cell.  Therefore $U_{\rho}^{\adj}$ and $U_{\rho}^{\face}$ are both nonempty and hence contractible by Lemma \ref{dlcontractlemma}, so the induced map of chain complexes is surjective, and hence an isomorphism.
\end{proof}

\bn  We will say that a quasi--linear PL--Morse function $W$ on a locally linear complex $X$ is \emph{$n$--acyclic} if $d_W^{\tot}(\sigma)$ is $(n-\dim(\sigma))$--acyclic for every $W$--constant cell $\sigma \subseteq X$.  By Lemma~\ref{dlvarlemma}, this is equivalent to saying that $d_W^{\face}(\sigma)$ or $d_W^{\adj}(\sigma)$ are $(n-\dim(\sigma))$--acyclic.  Let $\sigma$ be a $W$--constant cell, and let $d_W^{\face, \sus}(\sigma)$ denote the union of all cells $\tau$ such that:
\begin{itemize}
\item $\tau \subseteq \rho$ for some $\rho \in \calR_{\sigma}^{\max}$, and 
\item $\tau \cap \sigma \subsetneq \sigma$.
\end{itemize} This is called the \emph{suspended facial descending link}.  This terminology is motivated by the following result, which says that $d_W^{\face, \sus}(\sigma)$ has the connectivity of the $\dim(\sigma)$--fold suspension of $d_W^{\face}(\sigma)$.

\begin{lemma}\label{dlgluinglemma}
Let $X$ be a locally linear complex and $W$ a quasi--linear PL--Morse function on $X$.  Suppose that $W$ is $n$--acyclic.  Then for every $W$--constant $k$--cell $\sigma \subseteq X$, the complex $d_W^{\face, \sus}(\sigma)$ is $n$--acyclic.
\end{lemma}

\begin{proof}
If $W$ is sharp, then $\dim(\sigma) = 0$, so this follows by definition.  Therefore, we may assume that $W$ is linear.  As in the proof of Lemma \ref{dlvarlemma}, recall the set $\calR_{\sigma}^{\max}$.  For each $\rho \in \calR_{\sigma}^{\max}$, let $U_{\rho} = d_W^{\face, \sus}(\sigma) \cap \rho$ and let $U'_{\rho} = U_{\rho} \cap d_W^{\face}(\sigma)$.  From Lemma~\ref{dlcontractlemma}, we see that for any $\rho_1,\ldots, \rho_n \in \calR_{\sigma}^{\max}$ with $U'_{\rho_1} \cap \ldots \cap U'_{\rho_n}$ non--empty, then $U'_{\rho_1} \cap \ldots \cap U'_{\rho_n} \simeq *$.  Furthermore, if $U'_{\rho_1} \cap \ldots \cap U'_{\rho_n} \neq \emptyset$, then $$U_{\rho_1} \cap \ldots \cap U_{\rho_n} \simeq U'_{\rho_1} \cap \ldots \cap U'_{\rho_n}$$ via a straight--line homotopy.  On the other hand, if $U'_{\rho_1} \cap \ldots \cap U'_{\rho_n} = \emptyset$, then $$U_{\rho_1} \cap \ldots \cap U_{\rho_n} = \partial \sigma.$$  Now, consider the topological space $\partial \sigma * d_W(\sigma)$.  Note that this is no longer a locally linear cell complex.  For each $\rho \in \calR_{\sigma}^{\max}$, let $V_{\rho}$ denote the topological space $U'_{\rho} * \partial \sigma$, and let $\calV = \{V_{\rho}: \rho \in \calR_{\sigma}^{\max}\}$.  Finally, let $X_{\rho} = \ConvJoin(\partial \sigma, U'_{\rho})$, and let $\calX = \{X_{\rho}\}_{\rho \in \calR_{\sigma}^{\max}}.$ Note that for each $\rho \in \calR_{\sigma}$, we have an inclusion $U_{\rho} \hookrightarrow X_{\rho}$ and a map $V_{\rho} \rightarrow X_{\rho}$.  Therefore we have maps of nerves $$\calN(\calU) \rightarrow \calN(\calX) \leftarrow \calN(\calV).$$ These maps are both homotopy equivalences, since for any $\rho_1,\ldots, \rho_n$, we have either $$U_{\rho_1} \cap \ldots \cap U_{\rho_n} \simeq V_{\rho_1} \cap \ldots \cap V_{\rho_n} \simeq X_{\rho_1} \cap \ldots \cap X_{\rho_n} \simeq *$$ or $$U_{\rho_1} \cap \ldots \cap U_{\rho_n} \simeq V_{\rho_1} \cap \ldots \cap V_{\rho_n} \simeq X_{\rho_1} \cap \ldots \cap X_{\rho_n} = \partial \sigma.$$ Therefore $d_W^{\face, \sus}(\sigma) \simeq \partial \sigma * d_W^{\face}(\sigma)$.  Then $\sigma$ is a compact convex set, so $\partial \sigma$ is homotopy equivalent to a sphere of dimension $\dim(\sigma) - 1$, so the lemma follows by Lemma~\ref{joinfact}. 
\end{proof}

\bn We now have the following lemma.

\begin{lemma}\label{plconvexfundamentallinear}
Let $X$ be a finite dimensional locally linear cell complex and let $W$ be a quasi--linear PL--Morse function.  Suppose that $W$ is $n$--acyclic.  Then the pair $(X, M(X))$ is $(n+1)$--acyclic.
\end{lemma}
\begin{proof}
Let
\begin{displaymath}
X_{m,k}\coloneqq W^{-1}([0,m)) \bigcup\{\sigma \text{ a cell of } X: \max\{W|_{\sigma}\} \leq m, \dim(\sigma \cap W^{-1}(m)) \leq k\}.
\end{displaymath}
\bn If $m > 0$, then $X_{m,k}$ is built out of $X_{m-1,k}$ iteratively by attaching, for each $W$--constant $k$--cell $\sigma$ with $W(\sigma) = m$, the space $s_{W}(\sigma)$ over $d_W^{\face, \sus}(\sigma)$.  This is given by attaching a contractible space over an $n$--acyclic space by Lemma~\ref{dlgluinglemma}, so $\HH_*(X_{m,k};\Z) \cong \HH_*(X_{m-1,k};\Z)$ for $* \leq n + 1$.
\end{proof}

\bn We now have the following lemma.

\begin{lemma}\label{plconvexapp}
Let $X$ be an $(n+1)$--acyclic, finite dimensional, locally linear cell complex and let $W$ be an $n$--acyclic linear PL--Morse function on $X$.  Assume that $W$ is quasi--linear.  Then $M(W)$ is $n$--acyclic.  
\end{lemma}

\begin{proof}
By the long exact sequence in relative homology for the pair $(X, M(W))$ and the fact that $\HH_k(X;\Z) = 0$ for $k \leq n +1$, the connecting homomorphism
\begin{displaymath}
\HH_{k+1}(X, M(W);\Z) \rightarrow \HH_k(M(W);\Z)
\end{displaymath}
\bn is an isomorphism for $0 \leq k \leq n$.  Hence Lemma~\ref{plconvexfundamentallinear} completes the proof.
\end{proof}

\section{The Complex of homologous curves}\label{homolcurvesection}

We now move on to the study of the complex of homologous curves $\cC_{\vec{x}}(S_g)$.  Our goal is to prove our main theorem (Theorem~\ref{homolcyclethm}), which says that $C_{\vec{x}}(S_g)$ is $(g-3)$--acyclic for $g \geq 2$.  

\begin{figure}[ht]
\includegraphics[scale = 0.6]{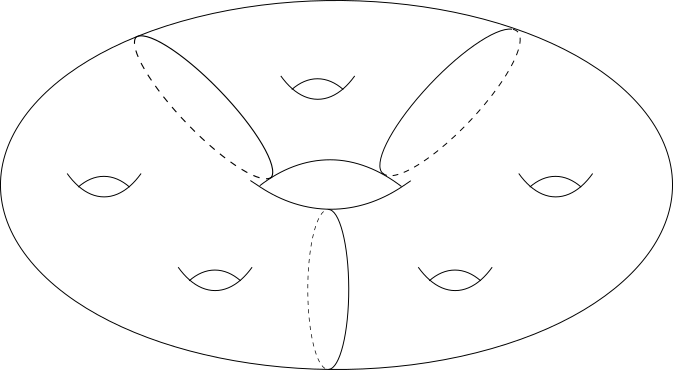}
\caption{A 2-cell in $\cC_{\vec{x}}(S_g)$.}
\label{homolcycleex}
\end{figure}

\p{The outline of the proof of Theorem~\ref{homolcyclethm}}  In Section~\ref{complexcycledefnsection}, we discuss the complex of minimizing cycles~\cite{BBM}.  We also discuss a linear PL-Morse function $W$ (in the sense of Section~\ref{plmorseconvexsection}) on the complex of minimizing cycles defined by Hatcher and Margalit~\cite{HatcherMargalithomologous}.  The min--set of $W$ will be $\cC_{\vec{x}}(S_g)$.  In Section~\ref{drainsubsection}, we show that an auxiliary complex called the complex of draining cycles is homologically spherical.  Instances of this complex will turn out to be the descending links of the PL--Morse function $W$.  We complete the proof of Theorem~\ref{homolcyclethm} in Section~\ref{homolcyclecompletesubsec} using the PL--Morse function $W$ and the results of Section~\ref{plmorseconvexsection}.

\subsection{The complex of minimizing cycles}\label{complexcycledefnsection}  We first define the complex of minimizing cycles, and then discuss Hatcher and Margalit's PL--Morse function on the complex of minimizing cycles.  Let $S = S_g$ and let $\vec{x}$ a primitive non-zero homology class in $\HH_1(S;\Z)$.  A \emph{basic cycle} for $\vec{x}$ is an oriented multicurve $M = a_1 \sqcup \ldots \sqcup a_k$ such that there is a unique collection of positive integers $\lambda_1,\ldots,\lambda_k$ with
\begin{displaymath}
\vec{x} = \sum_{i = 1}^k \lambda_i [a_i].
\end{displaymath}
\bn We will say that a multicurve $M = a_1 \sqcup\ldots \sqcup a_m$ is a \emph{cycle} if
\begin{enumerate}
\item each $a_i$ is a member of a basic cycle $M' \subseteq M$, and
\item  any non--trivial linear combination of the $[a_i]$ with nonnegative $\Z$--coefficients is nonzero.
\end{enumerate}
\p{Remark} Condition (2) is not present in Bestvina, Bux, and Margalit's original definition.  Gaifullin demonstrated that condition (2) needed to be included in the definition~\cite{Gaifullin}.  All of Bestvina, Bux and Margalit's~\cite{BBM} results still hold, since they implicitly assumed that condition (2) followed from (1).

\medskip
\bn Let $\mathfrak{S}$ be the set of isotopy classes of nonseparating simple closed curves in $S$ and let $\R^{\mathfrak{S}}$ be the real vector space consisting of finite linear combinations of elements of $\mathfrak{S}$.  If $M$ is a cycle, denote by $P_M \subseteq \R^{\mathfrak{S}}$ the convex hull of the basic cycles in $M$.  Let $\mathcal{M}$ denote the set of all cycles in $S_g$.  The complex of minimizing cycles $\cB_{\vec{x}}(S_g)$ is the union
\begin{displaymath}
\bigcup_{M \in \calM} P_M.
\end{displaymath}
\bn Observe that $\cB_{\vec{x}}(S_g)$ is a CW-complex where the $k$--cells correspond to cycles $M \in \calM$ with $\dim(P_M) = k$.  Bestvina, Bux and Margalit proved that $\cB_{\vec{x}}(S)$ is contractible~\cite{BBM}.  We will prove Theorem~\ref{homolcyclethm} by constructing a linear PL-Morse function $W$ on $\cB_{\vec{x}}(S_g)$ that is $(g-3)$--acyclic and has min-set equal to $\cC_{\vec{x}}(S)$. 

\p{The PL-Morse function}  We will use the PL-Morse function on $\cB_{\vec{x}}(S_g)$ originally defined by Hatcher and Margalit~\cite{HatcherMargalithomologous}.  Let $v$ be a vertex in $\cB_{\vec{x}}(S)$ represented by a basic cycle $M = a_0 \sqcup \ldots \sqcup a_m$.  By definition, $\vec{x} = \sum_{i \leq m} \lambda_i [a_i]$ for some unique positive integers $\lambda_i$.  We define
\begin{displaymath}
W(v) = \sum_{i \leq m} \lambda_i.
\end{displaymath}

\p{The min-set $M(W)$} There are no vertices $v \in \cB_{\vec{x}}(S_g)$ with $W(v) = 0$.  In fact, the lowest weight vertices are given by all $v$ with $W(v) = 1$.  The subcomplex of $\cB_{\vec{x}}(S)$ generated by $v$ with $W(v) = 1$ is precisely the complex of homologous curves $\cC_{\vec{x}}(S)$, so by we will let $M(W)$ denote the subcomplex of $\cB_{\vec{x}}(S)$ generated by vertices $v$ with $W(v) = 1$.

\subsection{The complex of draining cycles}\label{drainsubsection}
Our goal in this section is to prove Lemma~\ref{decconnlemma}, which says that a certain complex called the complex of draining cycles is highly acyclic.  This complex is a generalization of the complex of splitting curves introduced in Section~\ref{splitsection}.  We will leverage Lemma~\ref{decconnlemma} to prove that the PL-Morse function $W$ is $(g-3)$--acyclic.

\p{The complex of draining cycles}  We begin by defining a certain class of labeled surfaces.  Let $S = S_g^b$.  We say that $S$ is a \emph{partial cobordism} if it comes equipped with an orientation of some of the boundary components of $S$ and a partition of the oriented boundary components into two disjoint sets $\fB_+$ and $\fB_-$ such that the surface $S'$ given by gluing a disc to each nonoriented boundary component of $S$ is a cobordism from $\fB_+$ to $\fB_-$.  We will denote a partial cobordism by $\Sigma  = (S, \fB_+, \fB_-, \fB_0)$, where $\fB_0 = \pi_0(\partial S) \setminus (\fB_+ \cup \fB_-)$.  We say that a partial cobordism is:
\begin{itemize}
\item \emph{draining} if $|\fB_+| > |\fB_-|$;
\item \emph{balanced} if $|\fB_+| = |\fB_-|$; and 
\item \emph{flooding} if $|\fB_+| < |\fB_-|$.
\end{itemize}
\bn Let $S = S_g^b$ and let $\Sigma = (S, \fB_+, \fB_-, \fB_0)$ be a partial cobordism.  A vertex of $\cC_{\dr}(\Sigma)$ is an oriented multicurve $M \subseteq S$ such that:
\begin{itemize}
\item $|\pi_0(S \cut M)| = 2$;
\item one connected component of $S \cut M$ is a partial cobordism $\Sigma_M$ from a subset of $\fB_+$ to a union of $M$ and a subset of $\fB_-$; and
\item the partial cobordism $\Sigma_M$ is draining.  
\end{itemize}
\bn Note that the partial cobordism $\Sigma_M$ is unique.  Two examples of vertices in this complex can be found in Figure~\ref{deccycleex}.
\begin{figure}[ht]
\begin{tikzpicture}
\node[anchor=south west, inner sep = 0] at (0,0){\includegraphics[scale = 0.8]{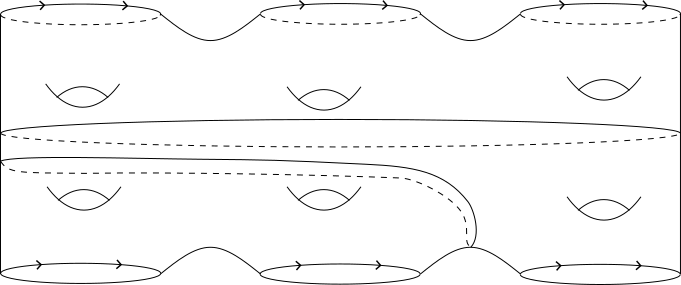}};
\node at (7.4,-0.5){\large $\fB_+$};
\node at (7.4,6.3){\large $\fB_-$};
\node at (10.2,2){\large $M$};
\node at (10.2,3.8){\large $N$};
\end{tikzpicture}
\caption{Two vertices $M$ and $N$ in the complex of draining cycles.}
\label{deccycleex}
\end{figure}
\bn We now define the higher-dimensional cells in the complex $\cC_{\dr}(\Sigma)$.  Let $\mathfrak{S}_{\Sigma}$ be set of oriented isotopy classes of essential simple closed curves on $\Sigma$.  We say that an oriented multicurve $M \subseteq S$ is \emph{representative} if: \begin{itemize}
\item each curve in $M$ is contained in a vertex $N$ of $\cC_{\dr}(\Sigma)$ with $N \subseteq M$;
\item if $N \subseteq M$ is a multicurve such that $\left| \pi_0(S \cut N)\right| = 2$ and at least one connected component $\Sigma_N$ of $S \cut N$ is a partial cobordism from a subset of $\fB_+$ to a union of $N$ and a subset of $\fB_-$, then $\Sigma_N$ is draining; and
\item any non--trivial linear combination of the homology classes represented by curves in $M$ with non--negative integer coefficients is nonzero in $\HH_1(S;\Z)$, where $S$ is the underlying surface of $\Sigma$.
\end{itemize}
\bn  Let $\mathfrak{S}_{\Sigma}$ denote the set of all oriented simple closed curves in $S$.  Let $\calM$ denote the set of representative multicurves in $\Sigma$.  Each vertex $M \in \calM$ corresponds to the point $a_1 + \ldots + a_k \in \R^{\mathcal{S}}$ with $M = a_1 \sqcup \ldots \sqcup a_k$.  We will henceforth identify a vertex $M \in \calM$ with this point in $\R^{\mathcal{S}}$.  For an arbitrary $M \in \calM$, let $P_M \subseteq \R^{\mathcal{S}_\Sigma}$ be the convex hull of the vertices supported in $M$.  The complex $\cC_{\dr}(\Sigma)$ is the union
\begin{displaymath}
\bigcup_{M \in \calM} P_M.
\end{displaymath}
\bn Observe that $\cC_{\dr}(\Sigma)$ is naturally a locally linear cell complex.  We will prove the following lemma.

\begin{lemma}\label{decconnlemma}
Let $\Sigma = (S_g^b, \fB_+, \fB_-, \fB_0)$ be a partial cobordism with $|\fB_+|, |\fB_-|\geq 2$ and $|\fB_+| \leq |\fB_-| +1$.  The complex $\cC_{\dr}(\Sigma)$ is $(g- 3+ |\fB_+|)$--acyclic.
\end{lemma}

\p{Capping boundary components} Let $S = S_g^b$ be a compact, orientable surface with $b \geq 1$.  Let $p \in \pi_0(S_g^b)$ be a boundary component, and let $S' = S_g^{b-1}$.  Let $\iota:S \rightarrow S'$ be an inclusion such that $\iota(p)$ bounds a disk.  Let $\cC'(S) \subseteq \cC(S')$ denote the full subcomplex generated by curves $c$ such that $\iota(c)$ is essential.  The map $\iota$ induces a pushforward $\Phi(\iota): \cC'(S) \rightarrow \cC(S)$.  Kent--Leininger--Schleimer show that this map $\Phi(\iota)$ is a homotopy equivalence~\cite[Corollary 7.2]{KLS} (their work uses punctures, but replacing boundary components with punctures does not change the isomorphism type of the curve complex).  In fact, their argument implies that $\Phi(\iota)$ is a fibration with fibers homeomorphic to trees (see~\cite[Proposition 7.6]{KLS}).  We extend this to the following.

\begin{lemma}\label{capequivlemma}
Let $S = S_g^b$ be a compact, orientable surface with $b \geq 1$.  Let $p \in \pi_0(S_g^b)$ be a boundary component, and let $S' = S_g^{b-1}$.  Let $\iota:S \rightarrow S'$ be an inclusion such that $\iota(p)$ bounds a disk.  Let $\cC'(S) \subseteq \cC(S')$ denote the full subcomplex generated by curves $c$ such that $\iota(c)$ is essential.  Let $X \subseteq \cC(S')$ be a subspace.  Then $\Phi(\iota)^{-1}(X) \rightarrow X$ is a homotopy equivalence.
\end{lemma}

\begin{proof}
As above, the work of Kent--Leininger--Schleimer says that $\Phi(\iota)$ is a fibration with contractible fibers.  By definition $\Phi(\iota)^{-1}(X)$ is the pullback of this fibration along the map $X \hookrightarrow \cC(S')$.  Therefore $\Phi(\iota)^{-1}(X)$ is also a homotopy equivalence.
\end{proof}

\bn This allows us to prove the following.

\begin{lemma}\label{decconlemmaaux}
Let $\Sigma = (S_g^b, \fB_+, \fB_-, \fB_0)$ be a partial cobordism with $|\fB_+|, |\fB_-| \geq 2$, $|\fB_0| \geq 1$.  Let $b \in \fB_0$ be a boundary component and let $\Sigma'$ be the partial cobordism given by gluing a disc along $p$.  The map $\cC_{\dr}(\Sigma) \rightarrow \cC_{\dr}(\Sigma')$ is a homotopy equivalence.
\end{lemma}
\begin{proof}
This follows from Lemma~\ref{capequivlemma}.  
\end{proof}
\bn We now prove Lemma~\ref{decconnlemma}.  The proof will proceed by inducting on the poset of partial cobordisms, denote $\parcob$.  The elements of $\parcob$ are partial cobordisms, and $\Sigma < \Sigma'$ if either $g(\Sigma) < g(\Sigma')$, or $g(\Sigma) = g(\Sigma')$ and $|\fB_+| < |\fB_+'|$
\begin{proof}[Proof of Lemma~\ref{decconnlemma}]
We will induct on the poset of partial cobordisms $\parcob$.  

\p{Base cases} The base cases are given by partial cobordisms with $|\fB_+| = 2$ and $|\fB_0| = 0$.  In this case, $\cC_{\dr}(\Sigma)$ is identified with the complex of splitting curves.  Lemma~\ref{gensplitconnlemma} says that when $|\fB_+| = 2$ and $|\fB_-| \geq 2$, the complex of splitting curves is $(g-1)$--acyclic, so the lemma holds.

\p{Induction on $\parcob$}  Let $\Sigma = (S_g^b, \fB_+, \fB_-, \fB_0) \in \parcob$ with $|\fB_+| \geq 3$, $|\fB_-| \geq 2$ and $|\fB_+| \leq |\fB_-| + 1$ such that the lemma holds for all $\Tau < \Sigma$ satisfying the hypotheses of the lemma.

Lemma~\ref{decconlemmaaux} says that we can fill in boundary components of $\fB_0$ with discs without changing the homotopy type of $\cC_{\dr}(\Sigma)$, so we may assume that $|\fB_0| = 0$.  Let $\Sigma' = (S_g^b, \fB_+', \fB_-, \fB_0)$ be the partial cobordism given by relabeling a boundary component $b \in \fB_+$ to be a boundary component of $\fB_0$ and let $\Sigma'' = (S_g^{b-1}, \fB_+', \fB_-)$ be the partial cobordism given by filling in $b$ with a disc.  By Lemma~\ref{decconlemmaaux}, there is a homotopy equivalence
\begin{displaymath}
\cC_{\dr}(\Sigma') \simeq \cC_{\dr}(\Sigma'').
\end{displaymath}
\bn By the inductive hypothesis, $\cC_{\dr}(\Sigma'')$ is $(g - 4 + |\fB_+|)$--acyclic and hence $\cC_{\dr}(\Sigma')$ is $(g-4 + |\fB_+|)$--acyclic as well.  Hence it suffices to show that $\mathfrak{c}(\cC_{\dr}(\Sigma')) +1 \leq \mathfrak{c}(\cC_{\dr}(\Sigma))$.  We do this in three steps:
\begin{enumerate}
\item We add to $\cC_{\dr}(\Sigma')$ any vertices $M \in \cC_{\dr}(\Sigma)$ that become inessential when $b$ is capped with a disc, and show that this strictly increases connectivity by 1.
\item We add vertices $M \in \cC_{\dr}(\Sigma)$ such that $M$ is a single simple closed curve, and show that this does not decrease connectivity.
\item We add in the remaining vertices and show again that connectivity does not decrease.
\end{enumerate}

\p{Adding curves that become inessential when $b$ is filled in with a disc}  Let $K$ be the set of vertices in $\cC_{\dr}(\Sigma)$ such that at least one curve becomes inessential when $b$ is filled in with a disc.  The set $K$ is a discrete set in the sense that no two vertices of $K$ share an edge.  Furthermore, every element of $K$ is a single curve $\gamma$ such that $\gamma$ surrounds $b$ and one other boundary component in $\fB_+$.  Let $\cC_{\dr}(\Sigma', K)$ denote the full subcomplex of $\cC_{\dr}(\Sigma)$ generated by $\cC_{\dr}(\Sigma')$ and $K$.  An example of a curve $M \in K$ can be found in Figure~\ref{addcurvedecex1}.  
\begin{figure}[ht]
\begin{tikzpicture}
\node[anchor=south west, inner sep = 0] at (0,0){\includegraphics[scale = 0.8]{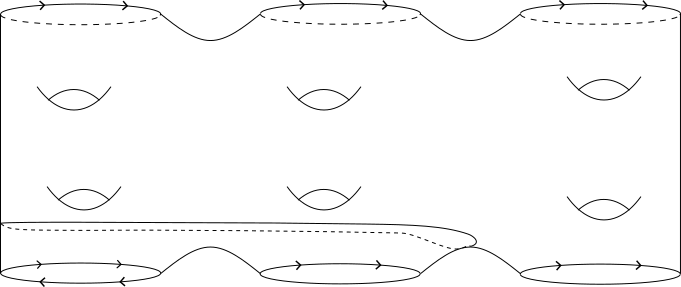}};
\node at (2,-.2){\large $b$};
\node at (4,1.6){\large $M$};
\node at (7.4,-0.5){\large $\fB_+$};
\node at (7.4,6.3){\large $\fB_-$};
\end{tikzpicture}
\caption{$M$ is a vertex in $K$.}
\label{addcurvedecex1}
\end{figure}

Now, observe that the subcomplex of $\cC_{\dr}(\Sigma',K)$ spanned by vertices disjoint from any $M \in K$ is isomorphic to $\cC_{\dr}(\Sigma'')$, which by Lemma~\ref{decconnlemma} is homotopy equivalent to $\cC_{\dr}(\Sigma')$.  Then the set $K$ is discrete and non--empty, so by a result of Brendle--Broaddus--Putman~\cite[Lemma 4.1]{BBP}, there is a homotopy equivalence
\begin{displaymath}
\cC_{\dr}(\Sigma',K) \simeq \cC_{\dr}(\Sigma') *K.
\end{displaymath}
\bn Hence $\mathfrak{c}(\cC_{\dr}(\Sigma', K)) \geq (g- 3 + |\fB_+|)$ by Lemma~\ref{joinfact} and the inductive hypothesis.

\p{Adding in single curves}  Let $C^1_{\dr}(\Sigma', K)$ denote the subcomplex of $\cC_{\dr}(\Sigma)$ consisting of multicurves $M$ such that every vertex $N \subseteq M$ either has $N \in \cC_{\dr}(\Sigma', K)$ or $N$ a single curve.  Let $M$ be a vertex of $C^1_{\dr}(\Sigma', K)$ and let $\Sigma_M$ denote the associated cobordism from the definition of the complex of draining cycles.  Let $\Tau_M$ denote the other acyclic component of $S_g^b \cut M$, which is naturally a cobordism.  An example of such a vertex $M$ can be found in Figure~\ref{addcurvedecex2}.  
\begin{figure}[ht]
\begin{tikzpicture}
\node[anchor=south west, inner sep = 0] at (0,0){\includegraphics[scale = 0.8]{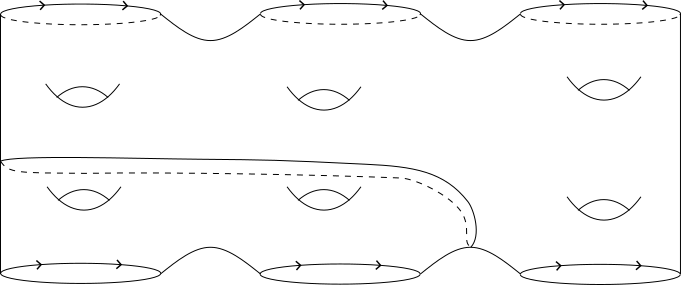}};
\node at (7.4,-0.5){\large $\fB_+$};
\node at (7.4,6.3){\large $\fB_-$};
\node at (2,-0.2){\large $b$};
\node at (4.2, 1.5){\large $\Sigma_M$};
\node at (4.2, 3.0){\large $M$};
\node at (10,4){\large $\Tau_M$};
\end{tikzpicture}
\caption{A curve $M \in C_{\dr}^1(\Sigma',K)$.  The top boundary components are in $\fB_-$, the bottom components are in $\fB_+$.}
\label{addcurvedecex2}
\end{figure}

For $M$ a vertex of $\cC_{\dr}^1(\Sigma', K)$, let
\begin{displaymath}
W(M) = 
\begin{cases}
0 & \text{if } M \in \cC_{\dr}(\Sigma', K)\\
g(\Sigma_M) & \text{otherwise}.
\end{cases}
\end{displaymath}
\bn We will show that $W$ is a $(g - 4 + |\fB_+|)$--acyclic PL--Morse function.  Since $\cC_{\dr}(\Sigma', K)$ is at least $(g- 3 +|\fB_+|)$--acyclic and $W$ is sharp (as in Section~\ref{plmorseconvexsection}), Lemma~\ref{plconvexapp} will imply that $\cC_{\dr}^1(\Sigma', K)$ is $(g - 3 + |\fB_+|)$--acyclic.  

Let $M \in \cC_{\dr}^1(\Sigma', K)$ be a positive weight vertex and let $N \in d_W^{\adj}(M)$ be a vertex.  Observe that $N$ is supported on either $\calT_M$ or $\Sigma_M$.  If $N$ is supported on $\calT_M$, then $N \in \cC_{\dr}(\calT_M)$.  Otherwise, $N \in \cC_{\dr}(\Sigma_M)$, which implies that $N$ is a singleton.  Hence there is a canonical isomorphism
\begin{displaymath}
d_W(M) \cong \cC_{\spl}(\Sigma_M) * \cC_{\dr}(\calT_M).
\end{displaymath}
\bn  By Proposition~\ref{gensplitconnlemma}, $\mathfrak{c}(\cC_{\spl}(\Sigma_M)) = g(\Sigma_M) - 2$.  Then by the inductive hypothesis, we have
\begin{displaymath}
\mathfrak{c}(\cC_{\dr}(\calT_M)) = g(\calT_M) - 3 + |\fB_+| - 1 = g(\calT_M) - 4 + |\fB_+|.
\end{displaymath} 
\bn Hence by Lemma~\ref{joinfact}, we have
\begin{displaymath}
\mathfrak{c}(d_W(M)) = g(\Sigma_M) - 2 + g(\calT_M) - 4 + |\fB_+| + 2 = g(\Sigma) - 4 + |\fB_+|
\end{displaymath}
\bn so this step is complete. 

\p{Adding other vertices}  We will now add in vertices in $\cC_{\dr}(\Sigma) \setminus \cC_{\dr}^1(\Sigma', K)$.  Let $M$ be a vertex of $\cC_{\dr}(\Sigma)$ and let $\Sigma_M$ be the associated partial cobordism.  Let 
\begin{displaymath}
W(M) = 
\begin{cases}
0& \text{if } M \in \cC^1_{\dr}(\Sigma', K)\\
\left|\chi(\Sigma_M)\right| & \text{otherwise.}
\end{cases}
\end{displaymath}
\bn We will show that $W$ is a sharp $(g- 4+|\fB_+|)$--acyclic PL--Morse function.  An application of Lemma~\ref{plconvexapp} will complete the proof.

\p{Claim} The PL--Morse function $W$ is sharp (as in Section \ref{plmorseconvexsection}).

 \p{Proof of claim} Let $P \subseteq \cC_{\dr}(\Sigma)$ be a cell with $\dim(P) \geq 1$.  Assume that there are $M, N \in P^0$ with $W(N) = W(M) \geq 1$.  Suppose by way of contradiction that $\max_{L \in P^0} W(L) = W(N)$.  Let $P' = N \cup M$, which is a sub--cell of $P$.  If $\Sigma_M \subseteq \Sigma_N$ then $W(M) < W(N)$, so assume that $\Sigma_M \not \subseteq \Sigma_N$.  Let $\Tau$ be a acyclic component of $\Sigma_M \cut P'$ which is not contained in $\Sigma_N$.  Therefore $\Tau_+ \subseteq N \cup \fB_+$ and $\Tau_- \subseteq M \cup \fB_-$.  Then $\Tau$ must be balanced.  Indeed, if $\Tau$ is draining, then the vertex $M'$ adjacent to $M$ corresponding to $\Tau$ cannot be a vertex of $\cC_{\dr}(\Sigma)$, since the cobordism $\Sigma_M \cut \Tau$ is not draining.  Likewise, if $\Tau$ is flooding, then the vertex $N'$ adjacent to $N$ could not be a vertex of $\cC_{\dr}(\Sigma)$, since the cobordism $\Sigma_N \cup \Tau$ cannot be draining.  But then $\Tau$ is balanced, which means that $N' = N \setminus (\Tau_+) \cup (\Tau_- \cap M)$ has $W(N'') \geq 1$.  By construction, we have $\Sigma_N \subseteq \Sigma_{N'}$, so $W(N') > W(N)$.  This contradicts our assumption that $W(N)$ was maximal on $P^0$. 

\p{Continuing to add other vertices} To see that $W$ is $(g- 4 + |\fB_+|)$--acyclic, let $M$ be a vertex in $\cC_{\dr}(\Sigma)$ of positive weight.  Observe that the cobordism $\Sigma_M = (S_M, \fB_+^M, \fB_-^M)$ satisfies the following two properties:
\begin{itemize}
\item $|\fB_+^M| = |\fB_-^M| + 1$; and
\item $p \in \fB_+^M$.
\end{itemize}
\bn Indeed, otherwise $M$ would be a vertex of $\cC_{\dr}(\Sigma', K)$.  Let $\Tau_M$ be the connected component of $\Sigma \cut M$ which is not $\Sigma_M$.  An example of such a vertex $M$ can be found in Figure~\ref{addcurvedecex3}.  
\begin{figure}[ht]
\begin{tikzpicture}
\node[anchor=south west, inner sep = 0] at (0,0){\includegraphics[scale = 0.8]{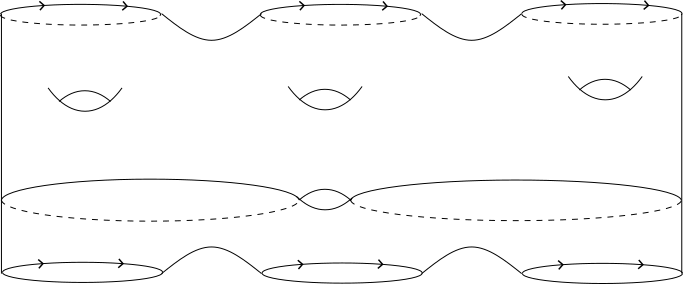}};
\node at (2,-.2){\large $b$};
\node at (6.8, 1){\large $\Sigma_M$};
\node at (6.8, 2.9){\large $\Tau_M$};
\node at (7.4,-0.5){\large $\fB_+$};
\node at (7.4,6.3){\large $\fB_-$};
\end{tikzpicture}
\caption{A vertex $M \in C_{\dr}^1(\Sigma',K)$ with $W(M) \geq 1$.  The top boundary components are in $\fB_-$, the bottom components are in $\fB_+$.}
\label{addcurvedecex3}
\end{figure}

The surface $\Tau_M$ is naturally a cobordism from a union of $M$ and a subset of $\fB_+$ to a subset of $\fB_-$.  We will show that $d_W^{\adj}(M) \cong \cC_{\dr}(\Sigma_M) * \cC_{\dr}(\Tau_M)$.  Let $N \in d_W^{\adj}(M)$ be a vertex.  Let $P = M \cup N$.  By the definition of $d_W^{\adj}$, $N$ and $M$ are connected by an edge, so $N$ is either supported on $\Sigma_M$ or $\Tau_M$.  If $N \subseteq \Sigma_M$, then $N \in \cC_{\dr}(\Sigma_M)$.  Any vertex $P \in \cC_{\dr}(\Sigma_M)$ must have $W(P) < W(M)$, so $d_W^{\adj}(M) \cap \cC_{\dr}(\Sigma_M) = \cC_{\dr}(\Sigma_M)$.  Otherwise, if $N \subseteq \Tau_M$, then $W(N) = 0$.  Therefore the partial cobordism $\Sigma_N$ realizing $N$ as draining is still draining even if $p$ is filled in with a disc.  Then since $\Sigma_N \supseteq \Sigma_M$, we have $N \in \cC_{\dr}(\Tau_M)$.

Hence $d_W(M)$ is at least $(g(\Sigma_M) - 3 + |\fB_+(\Sigma_M)| + g(\Tau_M) - 3 + |M| + |\Tau_M \cap \fB_+| + 2)$--acyclic by the inductive hypothesis and Lemma~\ref{joinfact}.  But then $g(\Tau_M) + g(\Sigma_M) + |M| = g$ and $|\fB_+^M| + |\Tau_M \cap \fB_+| = |\fB_+|$, so $d_W(M)$ is at least $(g - 4 + |\fB_+|)$--acyclic.  Therefore since $\cC_{\dr}^1(\Sigma', K)$ is $(g - 3 + |\fB_+|)$--acyclic, Lemma~\ref{plconvexapp} implies that $\cC_{\dr}(\Sigma)$ is $(g - 3 + |\fB_+|)$--acyclic, which completes the proof.
\end{proof}

\subsection{Completing the proof of Theorem \ref{homolcyclethm}}\label{homolcyclecompletesubsec}

We will prove in Lemma~\ref{linear} that the PL--Morse function $W$ defined in Section~\ref{complexcycledefnsection} is a linear PL--Morse function.  We then prove in Lemma~\ref{dllemmaconstruction} that the descending links of $W$--constant $k$--cells are given by joins of the complex of draining cycles.  We conclude by proving Theorem~\ref{homolcyclethm}, which we recall says that the complex of homologous curves $\cC_{\vec{x}}(S_g)$ is $(g-3)$--acyclic.

\begin{lemma}\label{linear}
The PL--Morse function $W$ is a linear PL--Morse function.
\end{lemma}
\begin{proof}
Let $\sigma \subseteq \cB_{\vec{x}}(S_g)$ be a cell corresponding to a multicurve $M$.  By definition, $\sigma$ is the convex hull of its vertices $v_0,\ldots, v_m$.  Hence every point $v \in \sigma$ is a linear combination
\begin{displaymath}
v = \sum_{i=0}^m t_iv_i
\end{displaymath}
\bn with $\sum_{i=0}^m t_i = 1$.  Set 
\begin{displaymath}
W^{\sigma}(v) = \sum_{i=0}^m t_i W(v_i).
\end{displaymath}
\bn It suffices to show that $W$ is well-defined on points in $\R^{\mathfrak{S}}$, i.e., independent of the choice $\sigma$.  To see this, suppose that 
\begin{displaymath}
v = \sum_{i=0}^m t'_iv_i
\end{displaymath}
\bn is another linear combination with $\sum_{i=0}^m t'_i = 1$.  If $a_0,\ldots,a_n$ are the underlying simple closed curves in the multicurve $M$ corresponding to $\sigma$, then each $v_i$ is by definition a formal sum
\begin{displaymath}
\sum_{j=0}^n \lambda_{i,j} a_j
\end{displaymath}
\bn such that
\begin{displaymath}
\vec{x} = \sum_{j=0}^n \lambda_{i,j} [a_j]
\end{displaymath}
\bn in $\HH_1(S_g;\Z)$.  Since
\begin{displaymath}
\sum_{i=0}^m t_iv_i = \sum_{i=0}^m t'_i v_i
\end{displaymath}
\bn there is a relation for each $a_j$ given by
\begin{displaymath}
\sum_{i=0}^m t_i\lambda_{i,j}= \sum_{i=0}^m t'_i\lambda_{i,j}.
\end{displaymath}
\bn Hence we have
\begin{displaymath}
\sum_{j = 0}^n \sum_{i=0}^m t_i\lambda_{i,j} =\sum_{j = 0}^n \sum_{i=0}^m t_i'\lambda_{i,j}.
\end{displaymath}
\bn But by the definition of $W$, this is precisely
\begin{displaymath}
\sum_{i=0}^m t_iW(v_i) = \sum_{i=0}^m t'_iW(v_i)
\end{displaymath}
\bn so the claim holds.
\end{proof}
\bn We now describe the $W$--constant cells of $\cB_{\vec{x}}(S_g)$.   
\begin{lemma}\label{drainmaxconstlemma}
Let $\sigma$ be a $W$--constant cell of $\cB_{\vec{x}}(S_g)$.  Let $M$ be the oriented multicurve corresponding to $\sigma$ and let $\Sigma_0 \sqcup \ldots \sqcup \Sigma_k$ be the connected components of $S_g \cut M$.  Each cobordism $\Sigma_i$ is balanced.
\end{lemma}
\begin{proof}
Suppose otherwise.  Then some cobordism $\Sigma_i$ is not balanced.  This cobordism $\Sigma_i$ is a cobordism between vertices $v$ and $w$ of $\sigma$.  If $\Sigma_i$ is unbalanced, then $W(v) \neq W(w)$ which contradicts the assumption that $\sigma$ is $W$--constant.
\end{proof}
\bn We now explicitly compute $d_W(M)$ in the case that $M$ is a $W$--constant $k$--cell.

\begin{lemma}\label{dllemmaconstruction}
Let $M \subseteq \cB_{\vec{x}}(S_g)$ be a $W$--constant $k$--cell. Let $\Sigma_0,\ldots, \Sigma_k$ be the connected components of $S_g \cut M$.  Then the natural inclusion
\begin{displaymath}
C_{\dr}(\Sigma_0) * \ldots * C_{\dr}(\Sigma_k) \rightarrow d_W^{\adj}(M)
\end{displaymath}
\bn is an isomorphism.
\end{lemma}
\begin{proof}
Let $N$ be a vertex in $d_W^{\adj}(M)$.  Let $\Sigma$ be a connected component of $S_g \cut M$, which is naturally a balanced cobordism by Lemma~\ref{drainmaxconstlemma}, such that $\Sigma \cap N \not \subseteq  \partial \Sigma$.   Let $\Tau \subseteq \Sigma$ be the cobordism between a subset of $\fB_+(\Sigma)$ and the union of the multicurve $N' = N \cap \Sigma$ and a subset of $\fB_-(\Sigma)$.  We will show that the cobordism $\Tau$ must be draining.  Let $P \subseteq M$ be a vertex containing $\fB_+(\Sigma)$, and let $P'$ be the unique vertex adjacent of $\cB_{\vec{x}}(S_g)$ that is both adjacent to $P$ and  contains $N'$.  Now, let
\begin{displaymath}
\sum_{p \in P} \lambda_p[p] = \vec{x}
\end{displaymath}
\bn be the linear relation corresponding to $P$.  By construction $\Tau$ realizes a relation in $\HH_1(S_g;\Z)$ of the form
\begin{displaymath}
\sum_{p \in \fB_+(\Sigma)} [p] = \sum_{n \in N'} [n]
\end{displaymath}
\bn where we identify $\fB_+(\Sigma)$ with its natural inclusion into $S_g$.  Then there is a linear relation in $\HH_1(S_g;\Z)$ supported on $\Tau$ given by
\begin{displaymath}
\sum_{p \in P} \lambda_p[P] - \min_{p \in \fB_+(\Sigma)}\lambda_p \left(\sum_{p \in \fB_+(\Sigma)}[p]\right) + \min_{p \in \fB_+(\Sigma)} \lambda_p \left(\sum_{n \in N'} [n] \right) = \vec{x}
\end{displaymath}
\bn Then this is another non--negative linear combination of curves in $P \cup P'$ which are equal in homology to $\vec{x}$.  By definition, this means the support of the above linear combination is contained in $P'$.  Since $N \in d_W^{\adj}(M)$, we must have $W(P') < W(P)$.  But this is only possible if $|\fB_+(\Sigma)| > |N'|$, so $\Tau$ is draining, as desired.
\end{proof}
\bn We now conclude Section~\ref{homolcurvesection} by proving that $\cC_{\vec{x}}(S_g)$ is $(g-3)$--acyclic.
\begin{proof}[Proof of Theorem~\ref{homolcyclethm}]  Let $M$ be a multicurve that represents a $W$--constant $k$-cell $\sigma$.  Label the connected components of $S \cut M$ by $\Sigma_0 \sqcup \ldots \sqcup \Sigma_k$.  By Lemma~\ref{dllemmaconstruction}, we have \begin{displaymath}
d_W^{\adj}(M) = \cC_{\dr}(\Sigma_0) * \ldots *\cC_{\dr}(\Sigma_k),
\end{displaymath}
\bn where $\Sigma_0,\ldots, \Sigma_k$ are the connected components of $S_g \cut M$.  By Lemma \ref{decconnlemma} and Lemma \ref{joinfact} we have that $d_W(\sigma)$ has connectivity at least $$-2+ \sum_{i = 0}^k g(\Sigma_i) + \left|\fB_+(\Sigma_i)\right| - 3 + 2.$$ Rearranging says that this sum equals $$-k - 3 + \sum_{i=0}^k g(\Sigma_i) + \left|\fB_+(\Sigma_i)\right|.$$  This latter summand is equal to $g$, and hence $d_W(\sigma)$ is at least $(g-3-k)$--acyclic.  Now, $W$ is a linear PL--Morse function by Lemma~\ref{linear}.  Hence $W$ is a $(g-3)$--acyclic linear PL-Morse function, so Theorem~\ref{homolcyclethm} follows by Lemma~\ref{plconvexapp} and Bestvina, Bux and Margalit's theorem that $\cB_{\vec{x}}(S_g)$ is contractible~\cite[Theorem E]{BBM}.  
\end{proof}

\bibliographystyle{alpha}
\bibliography{../../../../../Bibliography/mainbib}

\end{document}